\documentclass[11pt,twoside]{article}
\usepackage[english]{babel}
\usepackage{tikz}
\usepackage{multirow}
\usepackage{mathtools}
\usepackage{enumitem}
\usepackage{amsfonts}
\usepackage{booktabs}
\usepackage{amsmath}
\usepackage{amssymb}
\usepackage{comment}
\usepackage{makeidx}
\usepackage{amsthm,amscd, graphicx, environ}
\usepackage{color}
\usepackage{array}
\usepackage{xy}
\usepackage[font=small,labelfont=bf]{caption}
\usepackage{capt-of}
\usepackage{epsfig}
\usepackage{psfrag}
\usepackage{subfigure}
 \usepackage{colortbl}
  \definecolor{Lightgray}{RGB}{235,235,235}

\usepackage{xypic}
\input{xy}
\xyoption{arc}
\xyoption{all}
\definecolor{orange}{rgb}{0.698,0.133,0.133} 
\definecolor{green}{rgb}{0.33,0.42,0.18} 
\definecolor{greenf}{rgb}{0.13,0.55,0.13} 
\definecolor{newcolor}{rgb}{0.556364, 0.367273, 0.0763636}
\definecolor{colorfive}{rgb}{1.,0.27,0.}
\definecolor{colorone}{rgb}{0.70,0.13,0.}

\usepackage{varioref}
\usepackage{nameref}
 \usepackage[bookmarks=true,colorlinks=true]{hyperref}
\hypersetup{urlcolor=colorone, citecolor=colorone,  linkcolor=newcolor}

\makeatletter
\newcommand{\bdots}{\mathinner{\mkern1mu\raise\p@\vbox{\kern8\p@\hbox{.}}\mkern2mu\raise4\p@\hbox{.}\mkern2mu\raise7\p@\hbox{.}\mkern1mu}}
\makeatother

\newtheorem{thm}{Theorem}[subsection]
\newtheorem{defn}[thm]{Definition}

\newtheorem{prop}[thm]{Proposition}

\newtheorem{cor}[thm]{Corollary}
\newtheorem{lem}[thm]{Lemma}
\newtheorem{ur:remark}[thm]{Remark}

\newtheorem{ur:nota}[thm]{Notation}
\newenvironment{nota}{\begin{ur:nota}\rm}{\end{ur:nota}}
\newtheorem{ur:remarks}[thm]{Remarks}
\newenvironment{rems}{\begin{ur:remarks}\rm}{\end{ur:remarks}}
\newtheorem{ur:exemple}[thm]{Example}
\newenvironment{example}{\begin{ur:exemple}\rm}{\end{ur:exemple}}
\newtheorem{ur:examples}[thm]{Examples}

\newenvironment{bfen}{\begin{enumerate}[label=\bfseries{\arabic*.},ref=\bfseries{\arabic*.}]}{\end{enumerate}}
\newenvironment{ien}{\begin{enumerate}[label=(\roman*),ref=(\roman*)]}{\end{enumerate}}

\newdir^{ (}{{}*!/-5pt/@^{(}}

\newtheorem{pr}{}

\newcommand{\bpr}{\begin{pr} \begin{rm}}
\newcommand{\epr}{\end{rm} \end{pr}}

\newcommand{\nc}{{\mathbb C}}

\newcommand{\fraks}{{\mathfrak S}}
\newcommand{\nz}{{\mathbb Z}}

\def\ri{{\mathrm i}}

\newcommand{\p}[1]{\left({#1}\right)}

\newcommand{\av}[1]{\left|{#1}\right|}

\newcommand{\qu}[1]{\left[{#1}\right]}
\newcommand{\bm}[1]{\mbox{\boldmath${#1}$\unboldmath}}
\newcommand{\mb}[1]{\mathbf{{#1}}}
\newcommand{\brr}[1]{\left\{{#1}\right\}}

\def\Mat{{\mathrm{Mat}}}
\def\Id{{\mathrm{Id}}}
\def\Hom{{\mathrm{Hom}}}

\def\Sym{{\mathrm{Sym}}}

\def\expodot{\exp_{\odot}}
\def\odots{\odot\cdots\odot}

\newlength{\Oldarrayrulewidth}

 \usepackage{booktabs}

    {\end{list}}%
\begin{document}
\title{{\bf Formal first integrals and higher variational equations}}
\author{{\Large Sergi Simon} \\ \\
School of Mathematics and Physics, University of Portsmouth \\
Lion Gate Bldg, Lion Terrace 
Portsmouth PO1 3HF, UK\\
\texttt{sergi.simon@port.ac.uk}
}
\maketitle

\begin{abstract}

The question of how Algebra can be used to solve dynamical systems and
characterize chaos was first posed in a fertile mathematical context by Ziglin, Morales,
Ramis and Simó using differential Galois theory. Their study was aimed at first-order,
later higher-order, variational equations of Hamiltonian systems. Recent work by this
author formalized a compact yet comprehensive expression of higher-order variationals
as one infinite linear system, thereby simplifying the approach. More importantly, the
dual of this linear system contains all information relevant to first integrals, regardless of
whether the original system is Hamiltonian. This applicability to formal calculation of
conserved quantities is the centerpiece of this paper, following an introduction to the
requisite context. Three important examples, namely particular cases of Dixon's system, the SIR epidemiological model with vital dynamics and the Van der Pol oscillator, are tackled, and explicit convergent first integrals are provided for the first two.


\end{abstract}

\tableofcontents 

\section{Introduction}
Given an arbitrary dynamical system, the formulation of its higher variational equations as a linear (infinite) system $\p{\mathrm{LVE}^\star}$ has shown potential to make strong inroads in the study of integrability \cite{SimonLVE,SimonFRW}. The study, part of which we will explain more in detail in \S \ref{intro}, is twofold:
\begin{ien}
\item on one hand, the original set of equations $\p{\mathrm{LVE}}$ is amenable to the Ziglin-Morales-Ramis-Sim\'o non-integrability framework  whenever the system is \emph{Hamiltonian} and the first integrals whose existence is obstructed are \emph{meromorphic} (\cite{Mo99a,MoralesRamis,MoralesRamisII,MoRaSi07a} and a long assorted array of references derived therefrom, including \cite{SimonLVE,SimonFRW}); there is further, as of yet unfinished study in the direction of situations where the system is not Hamiltonian (\cite{HuangShiLi,LiShi1,LiShi2}).
\item \label{item2} on the other, the \emph{dual} system $\p{\mathrm{LVE}}^\star$ has jets of formal first integrals among its solutions (\cite{ABSW,SimonLVE}), which entails the possibility to furnish first integrals instead of finding obstructions to them; the advantage of this second approach is that the original system need not be Hamiltonian, and the formal first integrals, if convergent need not be meromorphic. The only difficulties in this case are computational in nature, namely in the context of resummation techniques.
\end{ien}

In the present work we will exploit the second item \ref{item2}, namely by describing a method to automatically produce Taylor terms of formal first integrals by way of an automatic, easily recursified sequence of quadratures. First we recount the minimal background exposition necessary, then we present the main results in \S \ref{filtersec} and finally we apply it to two simple examples to test its accuracy and usefulness as well as make novel statements about the integrability of the examples themselves.

\subsection{The algebraic study of integrability} \label{intro}

\subsubsection{Basics} \label{basics}
Let $\psi\p{t,\cdot}$ be the flow and ${\phi}\left(t\right)=\psi\p{t,\bm{x}}$ a particular solution of a given autonomous system
\begin{equation} \label{sys}
\dot{{\bm z}} = X \left( \bm{z}\right) , \qquad X:\nc^m\to \nc^m,
\end{equation}
respectively. The \textbf{variational system} of \eqref{sys} along ${\phi}$ has $\frac{\partial}{\partial{z}}\psi\p{t,{\phi}}$ as a fundamental matrix:
\begin{equation}
\label{VE}
\tag{$\mathrm{VE}_{\phi}$}\dot{Y} =  A_1 Y, \quad A_1\p{t}  :=  \left.\frac{\partial X}{\partial \bm{z}}\right|_{{\small {z}={\phi}\p{t}}}\in \mathrm{Mat}_n\p{K},
\end{equation}
$K=\nc\p{\bm\phi}$ being the smallest differential field (\cite[Def. 1.1]{SingerVanderput}) containing all entries of $\bm\phi\p{t}$. $\frac{\partial^k}{\partial{\small \bm z}^k}\psi\p{t,{\phi}}$ are multilinear $k$-forms appearing in the Taylor expansion of the flow along $\bm\phi$:
\begin{equation}\label{taylorexp}
\psi\p{t,{z}} = \psi\p{t,{\phi}} + \sum_{k=1}^\infty \frac1{k!}\frac{\partial^k \psi\p{t,{\phi}}}{\partial{z}^k} \brr{{z}-{\phi}}^k;
\end{equation}
$\partial^k_z\psi\p{t,{\phi}}$ also satisfy an echeloned set of systems, depending on the previous $k-1$ partial derivatives and usually called \textbf{order-$k$ variational equations} $\mathrm{VE}_\phi^k$ (\cite[p. 859-861]{MoRaSi07a}, \cite[Corollary 3]{SimonLVE}). 
Thus we have, \eqref{sys} given, a \emph{linear} system ${{\mathrm{VE}_\phi}}=:{{\mathrm{VE}_\phi^1}}=:{{\mathrm{LVE}_\phi^1}}$
and a family of \emph{non-linear} systems $\brr{{{\mathrm{VE}_\phi^k}} }_{k\ge 2}$.

\cite{SimonLVE} presented an explicit \emph{linearized} version ${\mathrm{LVE}_\phi^k}$, $k\ge 1$, by means of symmetric products $\odot$ of finite and infinite matrices based on already-existing definitions by Bekbaev, e.g. \cite{bekbaevSept2009}. This was done in preparation for the Ziglin-Morales-Ramis-Sim\'o (ZMRS) theoretical framework based on monodromy and differential Galois groups \cite{Mo99a,SingerVanderput,Ziglin}, but has other consequences as well. More specifically, our outcomes in \cite{SimonLVE} have two applications for system \eqref{sys}, \emph{Hamiltonian or not}: 
\begin{itemize}
\item full structure of ${\mathrm{VE}_{\phi}^k}$ and ${\mathrm{LVE}_{\phi}^k}$, i.e. \emph{recovering the flow}, which underlies the ZMRS theoretical corpus in practicality, albeit as a tool rather than as a goal;
\item  a byproduct is the full structure of dual systems $\p{{\mathrm{LVE}_{\phi}^k}}^\star$, i.e. \emph{recovering formal first integrals of \eqref{sys}} in ways which simplified earlier results in \cite{ABSW} significantly. 
\end{itemize}
As said in the introduction, results in the present paper are based on the second of these applications. For examples of the first application, see \cite[\S 6]{SimonLVE} or the bulk of \cite{SimonFRW}. See also \cite{martsim1,martsim2} for examples where the non-linearized $\mathrm{VE}_k$ were used.

\subsection{Symmetric products, powers and exponentials}\label{symsection}

\begin{nota}\rm \label{notalex} The conventions listed below were already introduced in \cite{ABSW,SimonLVE,SimonFRW}:
\begin{bfen} 
\item \emph{Multi-index modulus, arithmetic, order and lexicographic order}: for $\mathbf{i}=\left(i_1 ,\ldots , i_n\right)\in\nz_{\ge 0}^n$, $i=\left|\mathbf{i}\right|:=\sum_k i_k$; addition and subtraction are defined entrywise as usual; $\mathbf{i}\le \mathbf{j}$ means $ i_k \le j_k$ for every $k\ge 1$; $\mathbf{i}<_{\mathrm{lex}}\mathbf{j}$ if 
$ i_1 = j_1,\ldots,i_{k-1}=j_{k-1}$ and $i_k < j_k$ for some $k\ge 1$. 
\item \label{notalex3}
Whenever such derivation is possible, we define the 
	{\em lexicographically sifted differential of $F\p{z_1,\dots,z_m}$ of order $m$}   as the row vector
\begin{equation}\label{lexdef}
F^{(m)}\left(\bm z\right):=\mathrm{lex}\left(\frac{\partial^{m} F}{\partial z^{i_1}_1 \ldots \partial z^{i_n}_n }\left(\bm z\right)\right), \end{equation}
where $ \av{\mb{i}} = m$ and entries are ordered as per $<_{\mathrm{lex}}$ on multi-indices. For instance, for $n=2$ the first two differentials would be
\[
F^{(1)} = \p{\begin{array}{ccc} \frac{\partial F}{\partial z_1} & \frac{\partial F}{\partial z_2} & \frac{\partial F}{\partial z_3} \end{array}},\quad
F^{(2)} = \p{\begin{array}{cccccc} 
\frac{\partial^2 F}{\partial z_1^2} & 
\frac{\partial^2 F}{\partial z_1z_2} & 
\frac{\partial^2 F}{\partial z_1z_3} & 
\frac{\partial^2 F}{\partial z_2^2} & 
\frac{\partial^2 F}{\partial z_2z_3} & 
\frac{\partial^2 F}{\partial z_3^2} 
\end{array}}.
\]
\item We define 
$
d_{n,k} := \binom{n+k-1}{n-1} , \; D_{n,k} := \sum_{i=1}^k d_{n,i}. 
$
It is easy to check there are $d_{n,k}$ $k$-ples of integers in $\brr{1,\dots,n}$, and just as many
homogeneous monomials of degree $n$ in $k$ variables. 
\item Given integers $k_1,\dots,k_n\ge 0$, we define the usual multinomial coefficient as
{\small\[
\binom{k_1+\dots+k_n}{k_1,\dots,k_n}:=\binom{k_1+\dots+k_n}{\mb{k}}:=\frac{\p{k_1+\dots+k_n}!}{k_1!k_2!\cdots k_n!}.
\]}
For a multi-index $\mb{k}\in\nz_{\ge 0}^n$,  define $\mb{k!}:=k_1!\cdots k_n!$. For any two such $\mb{k},\mb{j}$, 
we define 
{\small\begin{equation} \label{newbinom} \binom{\mathbf{k}}{\mathbf{p}} := \frac{k_1!k_2!\cdots k_n!}{p_1!p_2!\cdots p_n!\p{k_1-p_1}!\p{k_2-p_2}!\cdots \p{k_n-p_n}!} 
=  \binom{k_1}{p_1} \binom{k_2}{p_2} \cdots \binom{k_n}{p_n},
\end{equation}}
and the multi-index counterpart to the multinomial,
$
\binom{\mb{k}_1+\dots+\mb{k}_m}{\mb{k}_1,\dots,\mb{k}_m}:=\frac{\p{\mb{k}_1+\dots+\mb{k}_n}!}{\mb{k}_1!\mb{k}_2!\cdots \mb{k}_n!}.
$
\end{bfen}
\end{nota}

The compact formulation called for by Notation \ref{notalex} (\textbf{3}) was achieved in \cite{SimonLVE} through an operation $\odot$ that had already been defined by Bekbaev (e.g.  \cite{bekbaevSept2009,bekbaevOct2010}) and was systematized using basic categorical properties of the tensor product. 
Let $K$ be a field and $V$ a $K$-vector space. Let $\Sym^rV$ be the symmetric power of $V$. We write $\bm{w}_1\odot \bm{w}_2$ for equivalence classes of tensor products of these vectors.
Hence,  product $\odot$ operates exactly like products of homogeneous polynomials in several variables.
\begin{nota}When dealing with matrix sets, we will use super-indices and subindices:
\begin{bfen}
\item The space of \textbf{$\p{i,j}$-matrices} $\Mat_{m,n}^{i,j}\p{K}$ is defined equivalently as the set of $d_{m,i}\times d_{n,j}$ matrices with entries in $K$, or vector space $\Hom_K\p{\Sym^jK^m;\Sym^iK^n}$. 
\item It is clear from the above that $\Mat_{n}^{0,0}\p{K}$ is the set of all scalars $\alpha\in K$ and $\Mat_{n}^{0,k}\p{K}$ (resp. $\Mat_{n}^{k,0}\p{K}$) is made up of all row (resp. column) vectors whose entries are indexed by $d_{n,k}$ lexicographically ordered $k$-tuples.
\item Reference to $K$ may be dropped and notation may be abridged if dimensions are repeated or trivial, e.g. $\Mat_{n}^{i,j}:=\Mat_{n,n}^{i,j}$, $\Mat_{m,n}^{i}:=\Mat_{m,n}^{i,i}$, $\Mat_n:=\Mat_n^{1}$, etcetera. 
\item $\Mat^{n,m}\p{K}$ denotes the set of block matrices 
$A =\p{A_{i,j}}_{i,j\ge 0}$ with $A_{i,j}:\Sym^{i}K^m\to \Sym^{j}K^n,$ hence $A_{i,j}\in \mathrm{M}_{d_{n,i}\times d_{m,j}} \p{K} = \mathrm{Mat}^{i,j}_{n,m}\p{K}$:
{\small\[
A = 
\p{\begin{tabular}{>{\centering\arraybackslash}m{.5cm}|>{\centering\arraybackslash}m{2.2cm}|>{\centering\arraybackslash}m{1.2cm}|>{\centering\arraybackslash}m{-0.3cm}|c}
$\ddots$  &$\vdots$  & $\vdots$ & $\vdots$ &  \\ \cline{1-5}
\multirow{5}{*}{$\cdots$}    & \multirow{5}{*}{$A_{2,2}$} & \multirow{5}{*}{$A_{2,1}$} & \multirow{5}{*}{}&\multirow{5}{*}{$\hspace{-0.45cm}\leftarrow A_{2,0}$} \\ 
  &  &  &  &  \\ 
    &  &  &  &  \\ 
      &  &  &  &  \\ 
        &  &  &  &  \\  \cline{1-5}
\multirow{3}{*}{$\cdots$}    & \multirow{3}{*}{$A_{1,2}$} & \multirow{3}{*}{$A_{1,1}$} & \multirow{3}{*}{}&\multirow{3}{*}{$\hspace{-0.45cm}\leftarrow A_{1,0}$} \\ 
  &  &  &  &  \\  
    &  &  &  &  \\  \cline{1-5}
$\cdots$    & $A_{0,2}$ & $A_{0,1}$ & & ${\hspace{-0.45cm}\leftarrow A_{0,0}}$ \\ 
\cline{1-5}
\end{tabular}
}
\]}We write $\Mat:=\Mat^{n,n}$ if $n$ is unambiguous. 
Conversely, $\mathrm{Mat}^{i,j}_{n,m}$ is embedded in $\Mat^{n,m}$ by identifying every matrix $A_{i,j}$ with an element of $\Mat^{n,m}$ equal to $0$ save for block $A_{i,j}$.
\end{bfen}
\end{nota}
\begin{defn}[Symmetric products of finite and infinite matrices]\label{defSymProd}\cite{bekbaevSept2009,SimonLVE}
\begin{ien}
\item Let $A\in \Mat_{m,n}^{i_1,j_1}\p{K}$, $B\in \Mat_{m,n}^{i_2,j_2}\p{K}$.
Given any multi-index $\mathbf{k}=\p{k_1,\dots,k_n}\in\nz^n_{\ge 0}$ such that $\av{\mathbf{k}} = k_1+\dots + k_n = j_1+j_2$, 
define $C:= A\odot B  \in \Mat_{m,n}^{i_1+i_2,j_1+j_2}$ by 
{\small\begin{equation} \label{SymProd}
C \p{\bm{e}_1^{\odot k_1}\cdots \bm{e}_n^{\odot k_n}} = \frac{1}{\binom{j_1+j_2}{j_1}}\sum_{ \mb{p}} \binom{\mathbf{k}}{\mathbf{p}} A\p{ \bm{e}_1^{\odot p_1}\cdots \bm{e}_n^{\odot p_n}} \!\odot \!
B\p{ \bm{e}_1^{\odot k_1-p_1}\cdots \bm{e}_n^{\odot k_n-p_n}},
\end{equation}}notation abused by removing $\odot$ to reduce space, binomials as in \eqref{newbinom}
and summation taking place for specific multi-indices $\mathbf{p}$, namely those such that
\[ \av{\mathbf{p}} = j_1 \qquad \mbox{ and \qquad } 0 \le p_i\le k_i, \quad i=1, \dots , n. \]
\item For any $A,B\in \Mat^{n,m}\p{K}$, define $A\odot B = C\in \Mat^{n,m}\p{K}$ by \begin{equation}\label{SymProdInf} 
C=\p{C_{i,j}}_{i,j\ge 0}, \qquad C_{i,j} = \sum_{0\le i_1\le i,\; 0\le j_1\le j} \binom{j}{j_1}  A_{i_1,j_1}\odot B_{i-i_1,j-j_1}. 
\end{equation}
Same as always, ${}^{\odot k}$ will stand for powers built with this product. 
\end{ien}
\end{defn}

The following is a mere exercise in induction:
\begin{lem}\label{original} Defining $\bigodot_{i=1}^rA_i$ recursively by $\p{\bigodot_{i=1}^{r-1}A_i}\odot A_r$ with 
$A_i\in\Mat_{m,n}^{k_i,j_i}$, 
\begin{equation}  \label{multipleproduct}
\p{A_1\odots A_r}\bm{e}^{\odot\mb{k}} = \frac{1}{\binom{j_1+\dots+j_r}{j_1,j_2,\dots,j_r}}\sum_{\mb{p}_1,\dots,\mb{p}_r}\binom{\mb{k}}{\mb{p}_1,\dots,\mb{p}_r}\bigodot_{i=1}^r A_i \bm{e}^{\odot\mb{p}_i},
\end{equation}
if $ \av{\mb{k}}=j_1+\dots+j_r$, sums obviously taken for $\mb{p}_1+\dots+\mb{p}_r=\mb{k}$ and $\av{\mb{p}_i}=j_i$, for every $i=1,\dots,r$. \quad $\square$
\end{lem}

The following is straightforward and has already been seen e.g. in \cite{bekbaevSept2009,bekbaevOct2010,SimonLVE}, or can be easily derived therefrom:
\begin{prop}\label{odotproperties}
For any $A$, $B$, $C$, and whenever products make sense, 
\begin{enumerate}
\item $A\odot B = B\odot A$.
\item $\p{A+B} \odot C = A \odot C + B \odot C$. 
\item $\p{A\odot B} \odot C = A\odot \p{B\odot C}$. 
\item $\p{\alpha A}\odot B = \alpha \p{A\odot B}$ for every $\alpha \in K$. 
\item If $A$ is square and invertible, then $\p{A^{-1}}^{\odot k} = \p{A^{\odot k}}^{-1}$.
\item $A\odot B =0$ if and only if $A=0$ or $B=0$.
\item If $A$ is a square $(1,1)$-matrix, then 
$ A\bm{v}_1 \odot A\bm{v}_2 \odot \cdots \odot A\bm{v}_m = A^{\odot m} \bm{v}_1 \odot \cdots \odot \bm{v}_m . $
\item If $\bm{v}$ is a column vector, then 
$\p{A\odot B} \bm{v}^{\odot \p{p+q}} =  \p{A\bm{v}^{\odot p}} \odot \p{B\bm{v}^{\odot q}}$, $p,q\in \nz_{\ge 0} . $
\end{enumerate}
\end{prop}

\begin{lem}[\cite{SimonLVE}]\label{anyprodslem}
\begin{ien}
\item Given square $A,B\in\Mat_n^{k,k}$ and matrices $X_i\in\Mat_n^{k,j_i},\,i=1,2$,  
\begin{equation}\label{anyprods}
\p{A\odot B}\p{X_1\odot X_2} = \frac{1}{2} \p{AX_1\odot BX_2 + BX_1\odot AX_2},
\end{equation} 
and in general for any square $A_1,\dots,A_m\in\Mat_n^{k,k}$ and any $X_i\in\Mat_n^{k,j_i}$, $i=1,\dots,m$,
\begin{equation}\label{anyprodsgeneral}
\p{\bigodot_{i=1}^m A_i}\p{\bigodot_{i=1}^m X_i} = \frac{1}{m!} \sum_{\sigma\in\fraks_k}\bigodot_{i=1}^mA_{\sigma\p{i}}X_i. \qquad \square 
\end{equation} 
\item Given $A\in\Mat_n^{1,j}$ and $X_1,\dots,X_m$ such that $X_i\in\Mat_n^{1,q_i}$, $1\le j \le m$, 
{\small\begin{equation}\label{prodsmoregeneral}
\binom{m}{j} \!\p{A\odot\Id^{\odot m-j}_{n}}\!\bigodot_{i=1}^m X_i = \sum_{1\le i_1<\dots<i_j\le m}\!\qu{A\p{X_{i_1}\odots X_{i_j}}}\odot \!\bigodot_{s\ne i_1,\dots,i_j} X_s. \quad \square
\end{equation} }
\item Given a square matrix $A\in\Mat_n^{1,1}$ and $X_1,\dots,X_m$ such that $X_i\in\Mat_n^{1,j_i}$,
\begin{equation}\label{prodsgeneral}
\p{A\odot\Id^{\odot m-1}_{n}}\p{\bigodot_{i=1}^m X_i} = \frac{1}{m} \sum_{i=1}^m\p{AX_i}\odot \p{X_1\odots \widehat{X_i}\odots X_m}. 
\end{equation}
\item Given square matrix $X\in \Mat_n^{k-1,k-1}$ and vector $\bm{v}\in K^n$, we have, for each $i=1,\dots,n$,
\begin{equation} \label{hello}
\p{X\odot \bm{e}_i^T} \p{\bm{v}\odot \Id_n^{\odot k-1}} = \frac{k-1}{k} X\p{\bm{v}\odot \Id_n^{\odot k-2}}\odot \bm{e}_i^T + \frac{v_i}{k} X
\end{equation}
\end{ien}
\end{lem}

\begin{lem}\label{odeprods} Let $\p{K,\partial}$ be a differential field.
\begin{ien}
\item\label{leibnizlem} For any given 
$X\in\Mat_n^{k_1,j_1}\p{K}$ and $Y\in\Mat_n^{k_2,j_2}\p{K}$, 
\begin{equation}\label{leibniz}
\partial \p{X\odot Y} = \partial\p{ X} \odot Y + X\odot \partial\p{Y}.\qquad \square
\end{equation}
\item If $Y$ is a square $n\times n$ matrix having entries in $K$ and $\partial Y = AY$, then 
\begin{equation} \label{Symkeq} \partial\; \Sym^k Y= k \p{A \odot \Sym^{k-1}\p{\Id_n}}\Sym^k Y.
\end{equation}
\item If $X\in \Mat_n^{1,j_1}$ and $Y\in \Mat_n^{1,j_2}$ satisfy systems $\partial X = AX+B_1$ and $\partial Y = AY+B_2$
with $A\in\Mat_n^{1,1},B_i\in\Mat_n^{1,j_i},
$
then symmetric product $X\odot Y$ satisfies linear system
\begin{equation} \label{Symskeq}
\partial \p{X\odot Y } = 2\p{A\odot \Id_{d_{n,k}}}\p{X\odot Y} +  \p{B_1\odot Y + B_2\odot X}.
\end{equation}
\item If $
\partial X_i = AX_i+B_i, \;  i=1,\dots,m,
$ with $X_i,B_i\in \Mat_n^{1,j_i}$,  $A\in \Mat_n^{1,1}$ 
then 
\begin{equation} \label{anySymskeq}
\partial \bigodot_{i=1}^m X_i = m\p{A\odot\Id_{d_{n,k}}^{\odot m-1}}\bigodot_{i=1}^m X_i + \sum_{i=1}^mB_i\odot\bigodot_{j\ne i} X_j . \qquad \square 
\end{equation}
\end{ien}
\end{lem}

\begin{defn} (See also \cite{bekbaevOct2010}) \label{defexp}
for every matrix $A\in \Mat^{n,m}$ we define the formal power series 
\[
\expodot A := 1 + A^{\odot 1} + \frac 12 A^{\odot 2} + \cdots = \sum_{i=0}^{\infty} \frac{1}{i!} A^{\odot i}. 
\]
Whenever $A=0$ save for a finite distinguished submatrix $A_{j,k}$, the abuse of notation $\expodot A_{j,k}=\expodot A$ will be customary.
\end{defn}
\begin{lem} \label{expprops0}
\begin{enumerate}
\item[\phantom{hola}]
\item For every two $A,B\in\Mat^{n,m}$, $\expodot\p{A+B} =\expodot{A}\odot\expodot{B}$. 
\item For every $Y\in \Mat^{n,m}$ and any derivation $\partial:K\to K$,  $\partial\expodot Y = \p{\partial Y}\odot \expodot Y. $
\item (\cite[Corollary 3]{bekbaevSept2009}) Given square matrices $A,B\in \Mat^{1,1}_n$, $\expodot AB = \expodot A \expodot B$. 
\item In particular, for every invertible square $A\in \Mat^{1,1}_n$, $\expodot A^{-1} = \p{\expodot A}^{-1}$. 
$\square$
\end{enumerate}
\end{lem}


For examples and properties, see \cite[\S 3.1]{SimonLVE}.

\subsection{Application to power series} \label{powerseriessection}

\begin{lem}[\protect{\cite[Lemma 3.7]{SimonLVE}}] \label{expformalseries}
If $F=F_1\times \dots \times F_m$ is a \emph{vector} power series, adequate $M_F^{1,i}\in\Mat_{m,n}^{1,i}\p{K}$ render
\[
\fbox{$F\p{\bm{x}} = M_F \expodot X$}\quad \mbox{ where }
M_F := \p{\begin{tabular}{ccc|c}
$\cdots$  & $M_F^{1,2}$ &  $M_F^{1,1}$ & $M_F^{1,0}$ \\\hline
$\cdots$  & $0$ &  $0$ & $0$ 
\end{tabular}}\in\Mat^{m,n}.
\]
Following Definition \ref{defexp}, write $F\p{\bm{x}} = M_F \expodot \bm{x}$ if it poses no clarity issue. $\square$
\end{lem}
Thus every formal power series $F\in K\qu{\qu{\bm{x}}}$, $\bm{x}=\p{x_1,\dots,x_n}$ can be expressed in the form 
$M_F\expodot{\bm{x}}$, where 
\[
M_F := \p{\begin{tabular}{ccc|c}
$\cdots$  & $M_F^{1,2}$ &  $M_F^{1,1}$ & $M_F^{1,0}$ \\\hline
$\cdots$  & $0$ &  $0$ & $0$ 
\end{tabular}}\in\Mat^{1,n}\p{K}, \qquad X := \p{\begin{tabular}{c|c}
  $0$     & $\bm{x}$ \\ \hline
  $0$ & $0$ 
\end{tabular}}.
\]

\begin{equation} \label{Fjet}
M_F = J_F+M_F^{1,0} := \p{\begin{tabular}{ccc|c}
$\cdots$  & $M_F^{1,2}$ &  $M_F^{1,1}$ & $0$ \\\hline
$\cdots$  & $0$ &  $0$ & $0$ 
\end{tabular}}
+  \p{\begin{tabular}{c|c}
   $0$ & $M_F^{1,0}$ \\\hline
   $0$ & $0$ 
\end{tabular}}.
\end{equation}
See \cite[\S 3.2]{SimonLVE} for more details.

\subsection{Higher-order variational equations}\label{structure}

See \cite{SimonLVE} for proofs and further elaboration on everything summarized below:

\begin{nota}\label{notationc}
For every set of indices $1\le i_1\le \dots \le i_r$ such that $ \sum_{j=1}^ri_j = k$, $c^k_{i_1,\dots,i_r}$ is defined as the amount of totally ordered partitions of a set of $k$ elements among subsets of sizes $i_1,\dots,i_r$:
{\small\begin{equation} \label{ck}
c_{i_1,\dots,i_j} = c^k_{i_1,\dots,i_j} = \frac{\binom{k}{i_1\,i_2\,\cdots\,i_j}}{n_1!\cdots n_m!}, \quad
\left\{ \!\begin{array}{l}\p{i_1,\dots,i_j} = \p{k_1\stackrel{n_1}{\dots}k_1,\cdots,k_m\stackrel{n_m}{\dots}k_m}, \\ 1\le k_1<k_2<\dots< k_m.\end{array} \right.
\end{equation}}
\end{nota}

\begin{prop}\label{nrVE}
Let $K=\nc\p{\bm\phi}$ and $A_k$, $ Y_k$ be  $\partial_k X\p{\bm\phi}$, $\partial^k_z \psi\p{t,\bm
\phi}$ minus crossed terms; let 
$A=J_X^\phi,Y=J_\phi$ be the derivative jets for $X$ and $\psi$ at $\bm\phi$, with $A_k$, $Y_k$ as blocks. Then, if $c^k_{\mathbf{i}}=\# \brr{\mbox{ordered $i_1,\dots,i_r$-partitions of $k$ elements}}$,
\begin{equation} \label{VEkredux} \tag{${\mathrm{VE}^k_\phi}$}
\dot{\overbracket[0.5pt]{Y_k}} =  \sum_{j=1}^k A_j \sum_{\av{\mathbf{i}} = k} c^k_{\mathbf{i}}{\small Y_{i_1}\odot Y_{i_2} \odot \cdots \odot Y_{i_j}},\quad \mbox{$k\ge 1$}.
\end{equation}
\end{prop}
\begin{lem}\label{AZproperties}
Let $Y\in\Mat\p{K}$ equal to zero outside of block row $_{1,k}$, $k\ge 1$:
\begin{equation} \label{Ydisplay}
Y:= \p{\begin{tabular}{cccc|c}
$\cdots$ &  $Y_3$ & $Y_2$ &  $Y_1$ & $0$ \\ \hline
$0$ &  $0$ & $0$ &  $0$ & $0$ 
\end{tabular}}, \qquad Y_i\in\Mat^{1,i}_{n}.
\end{equation}
Let $Z_{r,s}$, $s, r \ge 1$, be the corresponding block in $\expodot Y$. Then, 
\begin{enumerate}
\item Row block $r$ in $\expodot Y$ is recursively obtained in terms of row blocks $1$ and $r-1$:
\begin{equation}\label{expZrs}
Z_{r,s} = \frac1r \sum_{j=1}^{s-r+1}\binom{s}{j}Y_j\odot Z_{r-1,s-j}.
\end{equation}
In particular, $Z_{r,r}=Y_1^{\odot r}$ and $Z_{r,s}=0_{d_{n,r},d_{n,s}}$ whenever $r>s$. 
\item Using Notation \ref{notationc} and \eqref{ck}, for every $s\ge r$
\begin{equation} \label{Zrkdef} Z_{r,s}=\sum_{i_1+\dots+i_r=s} c^s_{i_1,\dots,i_r}Y_{i_1}\odot Y_{i_2} \odot \cdots \odot Y_{i_r}.
\end{equation}

\end{enumerate}
\end{lem}

\begin{nota}\label{notaAY}
$K:=\nc\p{\bm\phi}$, 
$A_i:=X^{\p{i}}\!\p{\bm{\phi}}$, $ Y_i := \mathrm{lex}\p{\frac{\partial^i }{\partial \bm{z}^i}\varphi{\p{t,\bm{\phi}}}}$ 
and, per Lemma \ref{AZproperties},
\begin{equation} \label{PhiLVE} 
\Upsilon_1=Y_1,\qquad \Upsilon_k =  \p{
\begin{tabular}{ccc}
$Z_{k,k}$ & \\
\cline{2-3} $Z_{k-1,k}$  & \multicolumn{2}{|c|}{}\\
\vdots  & \multicolumn{2}{|c|}{$\phantom{a}\Upsilon_{k-1}\phantom{a}$} \\
$Z_{1,k}$  & \multicolumn{2}{|c|}{} \\\cline{2-3} 
\end{tabular}
}, \quad k\ge 2,\end{equation}
formed by the first $k$ block rows and columns in $\Upsilon=\expodot Y$. Define $A,Y\in \Mat\p{K}$ as in Lemma \ref{AZproperties} with the above $A_i$, $Y_i$ as blocks.
Denote the canonical basis on $K^n$ (meaning the set of columns of $\Id_n$) by $\brr{\bm{e}_1,\dots,\bm{e}_n}$. 
\end{nota}

\begin{prop}[Explicit version of $\mathrm{LVE}_{\phi}^k$] \label{LVEprop} Still following Notation \ref{notaAY}, the infinite system
\begin{equation} \label{LVE}\tag{$\mathrm{LVE}_{\phi}$}
\fbox{$\dot X = A_{\mathrm{LVE}_{\phi}}X,$} \qquad\qquad A_{\mathrm{LVE}_{\phi}}:=A\odot \expodot \Id_n,
\end{equation}
has $\Upsilon:=\expodot Y$ as a solution matrix. Hence, for every $k\ge 1$, 
\begin{enumerate}
\item the lower-triangular recursive $D_{n,k}\times D_{n,k}$ form for $\mathrm{LVE}_{\phi}^k$ is $\dot{Y} = A_{\mathrm{LVE}_{\phi}^k}Y$, its system matrix being obtained from the first $k$ row and column blocks of $A_{\mathrm{LVE}_{\phi}}$:
\begin{equation} \label{ALVE}  A_{\mathrm{LVE}_{\phi}^k} =  \p{
\begin{tabular}{ccc}
$\binom{k}{k-1}A_1\odot \Id_n^{\odot {k-1}}$ & \\
\cline{2-3} $\binom{k}{k-2}A_2\odot \Id_n^{\odot k-2}$  & \multicolumn{2}{|c|}{}\\
\vdots  & \multicolumn{2}{|c|}{$A_{\mathrm{LVE}_{\phi}^{k-1}}$} \\
$\binom{k}{0} A_k$  & \multicolumn{2}{|c|}{} \\\cline{2-3} 
\end{tabular}
} ,\end{equation}
\item and the principal fundamental matrix for $\mathrm{LVE}_{\phi}^k$ is $\Upsilon_k$ from $\expodot Y$ in Notation \ref{notaAY}.
\end{enumerate}
\end{prop}

The construction of an infinite matrix $\Upsilon_{\mathrm{LVE}_\phi}=\expodot Y$ follows in such a way, that the first $d_{n,k}:=\binom{n+k-1}{n-1}$ rows and columns are a principal fundamental matrix for $\mathrm{LVE}_\phi^k$. The definition is recursive and amenable to symbolic computation. 
\begin{example}
For instance, for $k=5$ we have
\[
A_{\mathrm{LVE}_{\phi}^5}  = \p{
\begin{array}{ccccc}
5 A_1\odot \Id_n^{\odot 4} & & & & \\
10 A_2\odot \Id_n^{\odot 3}  & 4 A_1\odot \Id_n^{\odot 3}  & & & \\
10 A_3\odot \Id_n^{\odot 2}  & 6 A_2\odot \Id_n^{\odot 2} & 3 A_1\odot \Id_n^{\odot 2} & & \\
5 A_4 \odot \Id_n & 4 A_3 \odot \Id_n & 3 A_2 \odot \Id_n & 2 A_1 \odot \Id_n & \\
A_5 & A_4 & A_3 & A_2 & A_1 
 \end{array}
 },
\]
and the principal fundamental matrix $\Upsilon_5$ is 
{\small  \begin{equation} \label{PhiLVE5}
 \p{
\begin{array}{ccccc}
Y_1^{\odot 5} & & & & \\
10Y_1^{\odot 3}\odot Y_2 & Y_1^{\odot 4}  & & & \\
10Y_1^{\odot 2}\odot Y_3 + 15Y_1\odot Y_2^{\odot 2}& 6 Y_1^{\odot 2}\odot Y_2 & Y_1^{\odot 3} & & \\
10 Y_2\odot Y_3 + 5Y_1\odot Y_4 & 4 Y_1\odot Y_3 + 3 Y_2 \odot Y_2 & 3 Y_1\odot Y_2 & Y_1^{\odot 2} \\
Y_5 & Y_4 & Y_3 & Y_2 & Y_1 
  \end{array}
 },
\end{equation}}hence \eqref{VEkredux} for $k=5$ is the lowest row in $A_{\mathrm{LVE}_{\phi}^5}$ times the leftmost column in $\Upsilon_5 $.
\end{example}

\subsection{First integrals and higher-order variational equations}\label{firstintegrals}

Let $F:U\subseteq \nc^n\to \nc^n$ be a holomorphic function and $\bm{\phi}:I\subset \nc \to U$. Firstly, the flow $\varphi\p{t,\bm{z}}$ of $X$ admits, at least formally, Taylor expansion \eqref{taylorexp} along $\bm{\phi}$ which is expressible as
\begin{equation} \label{taylorphi}
\varphi\p{t,\bm{\phi}+\bm{\xi}} = \bm{\phi} + Y_1\bm{\xi}  + \frac12 Y_2 \bm{\xi}^{\odot 2} + \dots = \bm{\phi}+J_\phi \expodot\bm{\xi}, 
\end{equation}
where $J_\phi$ is the jet for flow $\varphi\p{t,\cdot}$ along $\bm\phi$, displayed as $Y$ in \eqref{Ydisplay} and defined in Notation \ref{notaAY} -- that is, the matrix
whose $\odot$-exponential $\Upsilon$ is a solution matrix for \eqref{LVE}. 
Secondly, the Taylor series of $F$ along $\bm{\phi}$
can be written, cfr. \cite[Lemma 2]{ABSW} and Notation \ref{notalex}, 
\begin{equation} \label{taylorabsw} F\p{\bm{y}+\bm{\phi}}=F\left(\bm{\phi}\right)+\sum^{\infty}_{m=1} \frac{1}{m!}\left\langle F^{(m)}\left(\bm{\phi}\right)\,,\, \Sym^m \bm{y}\right\rangle. 
\end{equation}
Lemma \ref{expformalseries} and \eqref{Fjet} trivially implies  \eqref{taylorabsw} can
be expressed as 
$F\p{\bm{y}+\bm{\phi}} = M^\phi_F \expodot \bm{y}$, where
\[
M^\phi_F = J^\phi_F+F^{\p{0}}\!\p{\bm{\phi}} := \p{\begin{tabular}{ccc|c}
$\cdots$ &   $0$ &  $0$ & $0$ \\ 
$\cdots$  & $F^{\p{2}}\!\p{\bm{\phi}}$ &  $F^{\p{1}}\!\p{\bm{\phi}}$ & $F^{\p{0}}\!\p{\bm{\phi}}$ \\\hline
$\cdots$  & $0$ &  $0$ & $0$ 
\end{tabular}}\in\Mat^{1,n}\p{K},
\]
i.e. $J_F^\phi$ is the jet or horizontal strip of lex-sifted partial derivatives of $F$ at $\bm{\phi}$. 

Let 
 \begin{equation} \label{LVEstar}\tag{$\mathrm{LVE}_{\phi}^{\star}$}
\fbox{$\dot X = A_{\mathrm{LVE}^{\star}_{\phi}}X,$} \qquad\qquad A_{\mathrm{LVE}_{\phi}^\star}:=-\p{A\odot \expodot \Id_n}^T, 
\end{equation}
be the \textbf{adjoint} or \textbf{dual} variational system of \eqref{sys} along $\phi$.

It is immediate, upon derivation of equation $\Upsilon_k\Upsilon_k^{-1}=\Id_{D_{n,k}}$, that
\begin{lem}\label{adjoint}
$\p{\Upsilon^{-1}_k}^T$ is a principal fundamental matrix of $\p{\mathrm{LVE}_{\phi}^k}^\star$, $k\ge 1$, hence $\lim_k\p{\Upsilon_k^{-1}}^T$, is a solution to \eqref{LVEstar}. 
\end{lem}
The following was proven in \cite{MoRaSi07a} and recounted in \cite[Lemma 7]{ABSW}, and was may now be expressed in a simple, compact fashion:
\begin{lem} \label{lemmanewproof}
Let $F$ and $\bm{\phi}$ be a holomorphic first integral and a non-constant solution of \eqref{sys} respectively. Let $V:=J_F^T$ be  the transposed jet of $F$ along $\bm{\phi}$. Then, $V$ is a solution of \eqref{LVEstar}.
\end{lem}
Hence the issue, already considered in \cite{ABSW}, is whether some converse holds and a solution of the dual system $\left(\mathrm{LVE}^{k}_{\phi}\right)^{\star}$ is or can be the set of first $k$ terms of the formal (or perhaps even convergent) series form of a first integral. This was addressed in the aforementioned reference and was put in compact, explicit linearized form in \cite{SimonLVE}, and will be recounted below.

Define
\begin{equation} \label{jetA}
\widehat{A}:= \p{\begin{tabular}{ccc|c}
$\cdots$  &   $0$ &  $0$ &  $0$  \\
$\cdots$ &   $A_2$ &  $A_1$ & $A_0$ \\ \hline
$\cdots$ &  $0$ & $0$ &  $0$ 
\end{tabular}}, \qquad A_i:=X^{\p{i}}\!\p{\bm{\phi}}\in\Mat^{1,i}_n\p{K}, \qquad A_0:=X\p{\bm{\phi}}=\dot{\bm{\phi}}, 
\end{equation}
and let 
\begin{equation}\label{matA0}
\widehat{A}_{\mathrm{LVE}_{\phi}}:=\widehat{A}\odot \expodot \Id_n = \lim_k \widehat{A}_{\mathrm{LVE}^k_{\phi}} = \lim_k \p{\begin{tabular}{|c|c}
\cline{1-1} $\binom{k}{k}X^{\p{0}}\p{\bm{\phi}} \odot  \Id_n^{\odot k}$ & \\
\cline{1-2} $ \binom{k}{k-1}X^{\p{1}}\p{\bm{\phi}} \odot \Id_n^{\odot k-1}$ & \multicolumn{1}{c|}{\multirow{4}{*}{$ \widehat{A}_{\mathrm{LVE}^{k-1}_{\phi}} $}} \\
\cline{1-1} $\binom{k}{k-2}X^{\p{2}}\p{\bm{\phi}} \odot  \Id_n^{\odot k-2}$ & \multicolumn{1}{c|}{} \\
\cline{1-1} $\vdots$ &  \multicolumn{1}{c|}{} \\
\cline{1-1} $\binom{k}{0} X^{\p{k}}\p{\bm{\phi}} \odot \Id_n^{\odot 0}$ &  \multicolumn{1}{c|}{} \\
\hline
\end{tabular} 
}, 
\end{equation}
The following gave a compact form to {\cite[Th. 12]{ABSW}} in terms of $\odot$ and infinite matrices:
\begin{prop}[\protect{\cite[Prop. 6]{SimonLVE}}] \label{SeriesProp}
Let $F$, $\bm \phi$, $V=\p{\cdots \mid V_3\mid V_2 \mid V_1}$ as in Lemma \ref{lemmanewproof}. Then $\widehat{A}_{\mathrm{LVE}_{\phi}}^T V = 0$. More specifically, $\widehat{A}_{\mathrm{LVE}_\phi^{k-1}} \p{V_k\mid\cdots \mid  V_2 \mid V_1}=0$ for every $k\ge 1$, i.e.
\begin{equation} \label{kernelcondition}
\sum_{j=0}^{k-1} \binom{k-1}{j} \p{A_j\odot \Id_n^{k-1-j}}^T V_{k-j} = 0, \qquad \mbox{for every }k\ge 1. \quad\square
\end{equation}

\end{prop}
Hence, blocks in $V_1,\p{V_2,V_1}^T, \p{V_3,V_2,V_1}^T, \dots$ having all entries in the base field $K$ and satisfying both equations in Proposition \ref{SeriesProp} and \ref{lemmanewproof} are jet blocks $F^{\p{1}},F^{\p{2}},\dots$ of a formal series that will be a first integral if convergent. In other words:
\begin{defn}[\cite{ABSW}]
In the above hypotheses,
\begin{ien}
\item The vectors $\p{V_k,\dots,V_1}^T$ belonging to the intersection of $\ker\widehat{A}_{\mathrm{LVE}^{k-1}_{\phi}}^T$ and the solution subspace 
$\mathrm{Sol}_K \p{\mathrm{LVE}^k_{\phi}}^\star$ are called \textbf{admissible} solutions of the order-$k$ adjoint system.
\item We call the sum $\sum\limits_{k\ge 1} \frac{1}{k!} V_k \bm{z}^{\odot k} $ a \textbf{formal first integral} along $\phi$ if every finite truncation thereof $\sum\limits_{k =1}^m \frac{1}{k!} V_k \bm{z}^{\odot k} $ arises from an admissible solution $\p{V_m,\dots,V_1}$.
\end{ien}
\end{defn}

For a short introduction on how changes of variables are reflected on this scheme, as well as the explicit form for the \emph{monodromy matrix} \cite{Zoladek} of ${\mathrm{LVE}_{\phi}^k}$ along any path $\gamma$, see of \cite[\S 4 \& 5]{SimonLVE}.

\section{Jet filtering methods and computation of formal first integrals}
\label{filtersec}

\subsection{Preliminaries}

Let $f$ be a first integral of \eqref{sys}. The following are immediate: 
\begin{lem}\label{filteringout1}
\begin{ien}
\item For every $i\ge 1$, 
\begin{equation}
\p{fg}^{\p{i}} = \sum_{j=0}^i \binom{i}{j} f^{\p{j}}\odot g^{\p{i-j}}, 
\end{equation}
with the understanding that symmetric product by a scalar is identical to the usual product, hence
\begin{eqnarray*}
J\brr{fg} &=& U\odot \p{U^{\odot 2}}^T \qu{J\brr{f} \odot J\brr{g}} = U\odot \p{U^{\odot 2}}^T \qu{\expodot J\brr{f} \odot \expodot J\brr{g}} \\ &=&  U\odot \qu{U^T\expodot J\brr{f}}\odot\qu{U^T\expodot J\brr{g}}
\end{eqnarray*}
where $U = \p{\begin{tabular}{c|c}
   $0$ & $1$ \\\hline
   $0$ & $0$ 
\end{tabular}}$
\item 
For every $k\ge 1$, 
\begin{equation}
\p{f^k}^{\p{i}} = k!\sum_{j=1}^{k}\frac{f^{j-1}}{\p{j-1}!}\sum_{i_1+\dots+i_{k-j+1}=i}c^i_{i_1,\dots,i_{k-j+1}} \bigodot_{u=1}^{k-j+1}f^{\p{i_u}},
\end{equation}
hence defining 
\[
E_1 = \p{\begin{tabular}{c|c}
   $\bm{e}_1^T$ & $0$ \\\hline
   $0$ & $0$ 
\end{tabular}}, \qquad U = \p{\begin{tabular}{c|c}
   $0$ & $1$ \\\hline
   $0$ & $0$ 
\end{tabular}},
\]
we have 
\begin{eqnarray}
J\brr{f^k} 
&=& k! \brr{U\odot \p{U^{\odot k}}^T \qu{\expodot J\brr{f}}}=k!\brr{U\odot \p{U^{\odot k}}^T \qu{\expodot J\brr{f}}} \\
&=&  U\odot \qu{U^T\expodot J\brr{f}}^{\odot k}=U\odot \qu{U^TJ\brr{f}}^{\odot k}
\end{eqnarray}
\item 
For every $k\ge 1$, 
\begin{equation}
f^{k}\p{1/f^k}^{\p{i}} = k\sum_{j=0}^{i}\frac{\p{k+j-1}!}{k!}f^{-j}\p{-1}^j\sum_{i_1+\dots+i_j=i}c^i_{i_1,\dots,i_j} \bigodot_{u=1}^jf^{\p{i_u}},
\end{equation}
hence
\begin{eqnarray}
J\brr{1/f}&=&\frac{1}{f}U\odot\brr{\p{\begin{tabular}{c|c}
$\vdots$ & $\vdots$ \\
$0$ & $3!\p{-1/f}^3$ \\
$0$ & $2!\p{-1/f}^2$ \\
$0$ & $1!\p{-1/f}$ \\\hline
$0$ & $1$ 
\end{tabular}}^T\expodot \p{J\brr{f}-fU}} \\
&=&  U\odot \p{\frac{1}{f}U^TJ\brr{f}}^{\odot-1}=U\odot \p{\sum_{i=0}^\infty \p{-1}^i\qu{\frac{1}{f}U^TJ\brr{f}-\mathbf{1}}^{\odot i}}
\end{eqnarray}
\end{ien}
\end{lem} 

\subsection{Progressive filtering and computation of formal first integrals} \label{progressive}

Let us describe a method to start with a jet of valuation one and then proceed to compute each jet term of that particular function immediately in terms of the previous ones. 

\begin{lem} \label{lemma1}
Define \begin{equation}\label{f0i1} F_1:=F_{i,1}:=\Id_n-\frac{1}{X_i^0}\p{X^0\odot \bm{e}_i^T}.
\end{equation}
The $n-1$ non-zero rows of $F_1Y_1^{-1}$ are linearly independent admissible solutions of degree one.
\end{lem}
\begin{proof}
Let $F_{i,1}:=\Id_n-\frac{1}{X_i\p{\phi\p{0}}}\p{X\p{\phi\p{0}}\odot \bm{e}_i^T}$. We have $\widehat{A}_{\mathrm{LVE}_\phi^0}=A_0^T=X\p{\phi}^T$. And $\p{\Upsilon_1^{-1}}^T=\p{Y_1^{-1}}^T$, the fundamental matrix for $\mathrm{VE}^\star_1$. Thus checking that \eqref{kernelcondition} holds equates to checking whether $A_0^T \p{Y_1^{-1}}F_{i,1}^T = \bm{0}^T_n$, i.e. that $F_{i,1}Y_1^{-1}A_0$ equals a column of zeros. This is easy to prove by noticing that the basic chain rule and \eqref{sys} entails $\dot A_0=A_1 A_0$ on one hand, and on the other Lemma \ref{adjoint} for $k=1$ implies $\frac{d}{dt}Y^{-1} = -Y^{-1}A_1$. Both of these facts put together imply 
\[
\frac{d}{dt} \p{Y_1^{-1}A_0} = \dot{\overbracket[0.5pt]{Y_1^{-1}}}X+Y^{-1}\dot{\overbracket[0.5pt]{X}}
=-Y^{-1}A_1X+Y^{-1}A_1X=0
\]
and thus $\frac{d}{dt}\p{F_{i,1}Y_1^{-1}A_0} =F_{i,1}\frac{d}{dt}\p{Y_1^{-1}A_0} =0$ as well, which means $F_{i,1}Y_1^{-1}A_0$ is a constant vector. Setting $t=t_0$ and using the fact that $Y_1$ is a principal fundamental matrix (i.e. $Y_1\p{t_0}=\Id$) means that for $t=t_0$, abridging notation $X^0 = \p{X_1^0,\dots,X_n^0}=X\p{\phi\p{t_0}}$, 
\begin{eqnarray*}
F_{i,1}Y_1\p{t_0}^{-1}A_0 &=& F_{i,1} A_0 = X\p{\phi\p{t_0}} - \frac{1}{X_i\p{\phi\p{t_0}}}\p{X\p{\phi\p{t_0}}\odot \bm{e}_i^T}X\p{\phi\p{t_0}} 
\\ &=& \p{\begin{array}{c} X_1^0 \\ X_2^0 \\ \vdots \\ X_i^0 \\ \vdots \\ X_{n-1}^0 \\ X_n^0\end{array}} - 
\frac{1}{X_i^0}\p{ 
\begin{array}{ccc} \cline{1-1} \cline{3-3}\multicolumn{1}{|c|}{\multirow{7}{*}{$0_{n\times\p{i-1}}$}} & X_1^0 & \multicolumn{1}{|c|}{\multirow{7}{*}{$0_{n\times\p{n-i}}$}} \\
\multicolumn{1}{|c|}{} &  X_2^0& \multicolumn{1}{|c|}{} \\
\multicolumn{1}{|c|}{} & \vdots  & \multicolumn{1}{|c|}{} \\
\multicolumn{1}{|c|}{} & X_i^0 & \multicolumn{1}{|c|}{} \\
\multicolumn{1}{|c|}{} & \vdots & \multicolumn{1}{|c|}{} \\ 
\multicolumn{1}{|c|}{} & X_{n-1}^0 & \multicolumn{1}{|c|}{} \\
\multicolumn{1}{|c|}{} & X_{n}^0 & \multicolumn{1}{|c|}{} \\
\cline{1-1} \cline{3-3}
\end{array}}\p{\begin{array}{c} X_1^0 \\ X_2^0 \\ \vdots \\ X_i^0 \\ \vdots \\ X_{n-1}^0 \\ X_n^0\end{array}} =\bm{0}_n.
\end{eqnarray*}
Still for $k=1$, we need to check that the $n-1$ non-zero columns of $\p{F_1Y^{-1}}^T$ are linearly independent. This is a consequence of the fact that  
\[
F_1^T = \p{\begin{array}{ccccccc}
\cline{1-3} 
\multicolumn{3}{|c|}{ \multirow{3}{*}{$\Id_{i-1}$}} & 0& & & \\
\multicolumn{3}{|c|}{ } &\vdots & & & \\
\multicolumn{3}{|c|}{ } &0 & & & \\ \cline{1-3}
-\frac{X_1}{X_i} & \cdots & -\frac{X_{i-1}}{X_{i}} & 0 & -\frac{X_{i+1}}{X_{i}}& \cdots & -\frac{X_n}{X_i} \\ 
\cline{5-7}
& & & 0 & \multicolumn{3}{|c|}{ \multirow{3}{*}{$\Id_{n-i}$}}  \\
 & & & \vdots & \multicolumn{3}{|c|}{ }  \\
& & & 0 & \multicolumn{3}{|c|}{ }  \\ \cline{5-7}
\end{array}}
\]
thus in $\p{Y^{-1}}^TF_1^T$ column $i$ equals $\bm{0}_n$ and its remaining $n-1$ columns are the output of subtracting multiples of column $i$ in $\p{Y^{-1}}^T$ from its remaining (already independent) $n-1$ columns. 
\end{proof}

Let $\phi\p{t}$ be a particular solution of \eqref{sys} whose parametric expression can be represented with one differential equation involving one single function $x\p{t}$, e.g. \eqref{parts1}, \eqref{parts}, \eqref{theeq} in \S \ref{examples}. 
\begin{thm} \label{othermethod}
In the above hypotheses, let $f_1$ be one of the columns of $\p{Y_1^{-1}}F_1^T$. Define $X_k:=\p{Y_1^{\odot k}}^T$. Then each term generated by the following recursion 
\begin{equation} \label{newrecursion}
f_k:=- X_k^{-1} \int X_k\brr{\sum_{j=2}^k\binom{k}{j}\p{A_j\odot \Id_n^{\odot k-j}}^Tf_{k-j+1}}dx,
\end{equation}
is the degree-$k$ term of the expansion of a formal first integral
\[ f\p{\bm{z}} = f_1 \bm{z} + \frac{1}{2} f_2 \bm{z}^{\odot 2} + \frac{1}{3!} f_3 \bm{z}^{\odot 3} + \dots,
\] provided that 
the limits or constants of integration at each stage \eqref{newrecursion} are such that
{\small\begin{equation} \label{newrecursion2}
f_k\p{A_0\odot \Id_n^{\odot k-1}}+ \binom{k-1}{1}f_{k-1} \p{A_1\odot \Id^{\odot k-2}_n} +\dots +\binom{k-1}{k-2}f_2 \p{A_{k-2}\odot \Id_n} + f_1 A_{k-1} =0.
\end{equation}} 
\end{thm}
\begin{proof}
Most of the work has been done already by Proposition \ref{SeriesProp}, Lemma \ref{lemma1} and the properties intrinsic to $\odot$. 
Firstly, $X^{-1}_k=\qu{\p{Y_1^{-1}}^{\odot k}}^T$. Secondly, condition \eqref{newrecursion2} is nothing but a repetition of \eqref{kernelcondition}. In order to check the consequences of defining \eqref{newrecursion}, all it takes is to realize that the first $d_{n,k}\times d_{n,k}$ equations of \eqref{LVEstar} read 
$$\dot{\overbracket[0.5pt]{\qu{\p{Y_1^{-1}}^{\odot k}}^T}} = -\p{A_1\odot \Id_n^{\odot k-1}}^T \qu{\p{Y_1^{-1}}^{\odot k}}^T,$$
which means $\p{Y_1^{\odot k}}^T$ is a principal fundamental matrix thereof;  the homogeneous part of the equations satisfied by $f_{k}$ displays the same matrix, thus \eqref{newrecursion} is just plain variation of constants applied to what is said in Lemma \ref{lemmanewproof}. 
\end{proof}

\subsection{Filtering with a single infinite matrix product} \label{filtersection}
The above result provides a recursive method to compute formal first integrals of arbitrary autonomous systems, thereby summing up the immediate aims of the current paper. The question arises, however, as to whether a \emph{single} infinite filter matrix $\Phi=\expodot F$ exists, amenable to recursive computation but expressible in compact form and \emph{computed without any further quadratures}, which multiplied times the fundamental matrix of the dual systems (whose terms are also computable recursively, as seen above) yields an already-trimmed matrix comprised of linearly independent admissible solutions in a single swipe. 

Building such an infinite quadrature-free filter matrix beyond $k=1$ has a value that is more theoretical than practical in nature at this stage, since the computational inviability of ever-growing matrices and the effort of discarding the cross-products $g^{\p{i_1}}_{j_1}\odots g^{\p{i_m}}_{j_m} $ get in the way of approximating jets properly. So the question arises: why consider one such matrix? The answer resides in the recent trends in categorification of dynamical systems (e.g. \cite{ChenGirelli}) and the potential to transport the techniques known in reductions of linear differential systems and flat meromorphic connections to arbitrary dynamical systems.

In this Section we prove such a matrix exists for meromorphic and holomorphic systems.

\subsubsection{Statement of the main result on infinite matrix filtering}

Let $\phi\p{t}$ be a non-trivial solution for \eqref{sys} and $t_0$ such that $X_i\p{\phi\p{t_0}}\ne 0$ for some $i$. Let $\Upsilon_k$ be the principal fundamental matrix for $\p{\mathrm{LVE}_\phi^k}$ such that $\Upsilon_k\p{t_0}=\Id$ and $F_1:=F_{i,1}$ as in \eqref{f0i1}.
\begin{thm}\label{filterthm}
In the above conditions, assume $X$ is holomorphic on a neighborhood $U$ of $\phi\p{t_0}$ and $Y_1$ and $Y_1^{-1}$ are holomorphic on a neighborhood $V$ of $t_0$. 
For every $k\ge 1$, define the constant matrices
\begin{eqnarray}\label{Fkstatement}
\label{Fkstatement1}\hspace{-1.5cm}F_{k}&=&-\frac{1}{X^0_i} \qu{\sum_{j=0}^{k-2}\binom{k-1}{j}F_{j+1}\p{A^0_{k-j-1}\odot \Id^{\odot j}}} U_k
, 
\end{eqnarray}
where we use superscript ${}^0$  to denote values at $t=t_0$, 
\begin{equation} \label{Fkstatement2}
U_k=\qu{\sum_{j=0}^{k-1} \binom{k}{j+1} \p{-1}^{j} \p{\Id - F_1}^{\odot j} \odot \Id^{\odot k-1-j}  }\odot \bm{e}_i^T = \qu{\bigodot_{j=1}^{k-1}\p{\Id-\zeta_k^jF_1}}\odot \bm{e}_i^T,
\end{equation}
and
$\zeta_k=\exp\p{\frac{2\pi\ri}{k}}$ is the primitive root of unity. 
Let $\Phi_k$ be the lower square right $D_{n,k}$-block of 
\[
\Phi=\expodot F, \qquad \mbox{where} \quad 
F=\p{\cdots \mid F_3\mid F_2\mid F_1}. \]
Then, the last $n-1$ non-zero columns of $\p{\Upsilon_k^{-1}}^T\p{\Phi_k}^T$ are linearly independent admissible solutions $g_1^{\p{\cdot}},g_2^{\p{\cdot}}, \dots,g_{n-1}^{\p{\cdot}} \in\mathrm{Sol}\p{\mathrm{LVE}_\phi^k}^\star\cap \ker \widehat{A}^T_{\mathrm{LVE}^{k-1}_{\phi}}$, i.e. truncated jets of functionally independent formal first integrals $g_1,\dots,g_{n-1}$ of valuation $1$, defined on a neighborhood of $\phi\p{t_0}$.
\end{thm}
\begin{rems} \label{importantremarks}
\begin{bfen}
\item[]
\item \label{rem1} We already know all the columns of $\p{\Upsilon_k^{-1}}^T\p{\Phi_k}^T$ belong to $\mathrm{Sol}\p{\mathrm{LVE}_\phi^k}^\star$, because $\frac{d}{dt} \p{\Upsilon^{-1}}^T = A_{\mathrm{LVE}^\star} \p{\Upsilon^{-1}}^T$ where $\Upsilon=\expodot Y$ , and the above finite-order products are nothing but truncations of the right-product of $\p{\Upsilon^{-1}}^T$ by a constant matrix. We also know that the non-zero columns in $\p{Y_1^{-1}}^{T}F^T_1$ are linearly independent, from whence follows the independence of the full non-zero columns in higher orders.

Thus all that is needed to establish Theorem \ref{filterthm} is to prove that these truncations are also annihilated by left-multiplication by $\widehat{A}^T_{\mathrm{LVE}^{k-1}_{\phi}}$. This will be done in \S \ref{prooffilterthm}.
\item In light of Remark \ref{rem1}, Lemma \ref{lemma1} basically entails that the entire Theorem holds true at level $k=1$ and, as seen in \S \ref{progressive}, that was all we needed to compute independent jets of formal first integrals, independently of Theorem \ref{filterthm}.
\end{bfen}
\end{rems}

\subsubsection{Proof of Theorem \ref{filterthm}}\label{prooffilterthm}

As per Remark \ref{rem1}, let us now prove that matrices defined by \eqref{Fkstatement1} and \eqref{Fkstatement2} are indeed lower blocks of a filter matrix $\Phi$ reducing $\p{\Upsilon^{-1}}^T$ to the desired form (namely one such its lower right $D_{n,k}\times D_{n,k}$ blocks belong to $\ker \widehat{A}^T_{\mathrm{LVE}^{k-1}_{\phi}}$). In other words, that
\begin{equation} \label{eqfinal}
F\cdot \Upsilon^{-1} \cdot \widehat{A}_{\mathrm{LVE}_{\phi}}=F\cdot \Upsilon^{-1} \cdot \p{\widehat{A}\odot \expodot \Id_n} = 0,
\end{equation}
Our first step is proving the following.
\begin{prop} \label{prefinal}
Identity \eqref{eqfinal} is true at $t_0$, i.e. $F \cdot \left.\widehat{A}_{\mathrm{LVE}_{\phi}}\right|_{t=t_0}=0$.
\end{prop}
\begin{proof}
Obviously we are using the fact that $\left.\Upsilon\right|_{t=t_0} = \expodot\Id$, i.e. an infinite identity matrix.
All occurrences of $A_i$ from here onward in this proof will refer to constant matrices $A_i^0=X^{\p{i}}\p{\phi\p{t_0}}$.
Case $k=1$ has already been tackled paragraphs earlier; in general, we need to prove that 
{\small \begin{equation} \label{hyp}
F_k\p{A_0\odot \Id^{\odot k-1}} + \binom{k-1}{1} F_{k-1}\p{A_1\odot \Id^{k-2}} + 
 \dots + \binom{k-1}{k-2} F_{2}\p{A_{k-2}\odot \Id} +
F_{1}\p{A_{k-1}} =0,
\end{equation}} for every $k\ge 1$. 

First of all, bilinearity from Prop. \ref{odotproperties} and a simple induction argument imply, for every set of $n\times n$ matrices $M_1,\dots,M_m$,
\[
\bigodot_{i=1}^m \p{\Id+M_i} = \Id^{\odot m}_n+\sum_{j=1}^m \Id_n^{\odot m-1} \odot M_j + \sum_{1\le j_1< j_2 \le m} \Id_n^{\odot m-2} \odot M_{j_1}\odot M_{j_2}+\dots + M_1\odots M_m
\]
and thus setting $m=k-1$ and $M_i=-\zeta^i_kF_1$ very basic properties of cyclotomic polynomials imply  
\begin{equation} \label{newUk}
U_k=\qu{\bigodot_{i=1}^{k-1 } \p{\Id-\zeta_k^iF_i}}\odot \bm{e}_i^T = \p{\sum_{j=0}^{k-1} \Id_n^{\odot j} \odot F_1^{\odot k-1-j}}\odot \bm{e}_i^T
\end{equation}
We claim that 
\begin{equation} \label{Ukk}
U_k\p{A_0\odot \Id_n^{\odot k-1}} = X_i \Id_n^{\odot k-1}, \qquad \mbox{ for every }k\ge 1.
\end{equation}
Indeed, we have, using \eqref{hello} and the rest of the properties of $\odot$ (all of which can be traced back to Lemma \ref{original} or the universal property of $\odot$, \cite[\S 2.1]{SimonLVE}),  
\begin{eqnarray*}
U_k\p{A_0\odot \Id_n^{\odot k-1}} &=& \qu{\p{\sum_{j=0}^{k-1} \Id_n^{\odot j} \odot F_1^{\odot k-1-j}}\odot \bm{e}_i^T} \p{A_0\odot \Id_n^{\odot k-1}} \\
&=& \frac{k-1}{k} \p{\sum_{j=0}^{k-1} \Id_n^{\odot j} \odot F_1^{\odot k-1-j}}\p{A_0\odot \Id_n^{k-2}} \odot \bm{e}_i^T + \frac{X_i}{k} \p{\sum_{j=0}^{k-1} \Id_n^{\odot j} \odot F_1^{\odot k-j}} \\
&=& \frac{k-1}{k}\sum_{j=0}^{k-2} \frac{k-1-j}{k-1} F_1^{\odot j} \odot A_0 \odot \Id_n^{\odot k-2-j}\odot \bm{e}_i^T+\frac{X_i}{k} \p{\sum_{j=0}^{k-1} \Id_n^{\odot j} \odot F_1^{\odot k-j}} \\
&=& \frac{1}{k}\sum_{j=0}^{k-2} \p{k-1-j} F_1^{\odot j} \odot A_0 \odot \Id_n^{\odot k-2-j}\odot \bm{e}_i^T+\frac{X_i}{k} \p{\sum_{j=0}^{k-1} \Id_n^{\odot j} \odot F_1^{\odot k-j}} \\
&=& X_i \Id_n^{\odot k-1}.
\end{eqnarray*}

The rest follows from \eqref{Ukk}:
\begin{eqnarray*}
F_{k}\p{A_0\odot \Id_n^{\odot k-1}} &=& -\frac{1}{X_i} \qu{\sum_{j=0}^{k-2}\binom{k-1}{j}F_{j+1}\p{A_{k-1-j}\odot \Id_n^{\odot j}}}U_k\p{A_0\odot \Id_n^{k-1}} \\
&=& -\frac{1}{X_i} \qu{\sum_{j=0}^{k-2}\binom{k-1}{j}F_{j+1}\p{A_{k-1-j}\odot \Id_n^{\odot j}}}X_i\Id_n^{\odot k-1}  \\
&=& -\sum_{j=0}^{k-2}\binom{k-1}{j}F_{j+1}\p{A_{k-1-j}\odot \Id_n^{\odot j}}  \\
\end{eqnarray*}
which proves \eqref{hyp} true.
\end{proof}

Double induction, along with properties of $\odot$, easily proves that
\[
A_0\odot A_u \odot \Id^{\odot v} = \p{A_0\odot \Id^{\odot v+1}}\p{A_u\odot \Id^{\odot v}}, \qquad u,v \ge 0,
\]
and just as easily establishes the following consequence thereof:
\begin{lem} \label{Alem} 
Using the redundant notation
\[
A_0 =  \p{\begin{tabular}{cc|c}
$\cdots$  &     $0$ &  $0$  \\
$\cdots$ &     $0$ & $A_0$ \\ \hline
$\cdots$ &   $0$ &  $0$ 
\end{tabular}}, \]
and in the hypotheses of \eqref{jetA} and \eqref{matA0}, 
\begin{eqnarray*}
\dot{\overbracket[0.5pt]{\widehat{A}\odot \expodot \Id}} &=& \p{A\odot \expodot \Id}\p{A_0\odot \expodot \Id} - A_0\odot A\odot \expodot \Id \\
&=& \p{A\odot \expodot \Id}\p{A_0\odot \expodot \Id}-\p{A_0\odot \expodot \Id}\p{A\odot \expodot \Id},
\end{eqnarray*}
in other words, 
{\small \[
\binom{i+j}{i}\dot{\overbracket[0.5pt]{A_j \odot \Id^i}} = \binom{i+j+1}{i} \p{A_{j+1}\odot \Id^{\odot i}}\!\p{A_0\odot \Id^{\odot\p{i+j}}} -\binom{i+j}{i-1}\p{A_{j+1}\odot A_0\odot \Id^{\odot\p{i-1}}} ,\]}
for every $i,j\ge 0$. $\square$
\end{lem}

Let us study the successive derivatives of $C=F\cdot \Upsilon^{-1} \cdot \widehat{A}_{\mathrm{LVE}_{\phi}}$.
We know $\dot{\overbracket[0.5pt]{\Upsilon^{-1}}}=-\Upsilon^{-1}A_{\mathrm{LVE}_{\phi}}$ from
Lemma \ref{adjoint}. Using Lemma \ref{Alem},
{\small\begin{eqnarray*}
\dot{\overbracket[0.5pt]{C}} &=& 
{\small F \qu{-\Upsilon^{-1}A_{\mathrm{LVE}_{\phi}}\widehat{A}_{\mathrm{LVE}_{\phi}} +\Upsilon^{-1}A_{\mathrm{LVE}_{\phi}}\p{A_0\odot \expodot \Id} - \Upsilon^{-1}A_0\odot A\odot \expodot \Id} } \\
&=& F \Upsilon^{-1}\qu{-A_{\mathrm{LVE}_{\phi}}\widehat{A}_{\mathrm{LVE}_{\phi}} +A_{\mathrm{LVE}_{\phi}}\p{A_0\odot \expodot \Id} - A_0\odot A\odot \expodot \Id} \\
&=& - F \Upsilon^{-1}\qu{A_{\mathrm{LVE}_{\phi}}^2+A_0\odot A\odot \expodot \Id} = - \p{\expodot F} \Upsilon^{-1}\qu{A_{\mathrm{LVE}_{\phi}}^2+A_0\odot A_{\mathrm{LVE}}} \\
&=&  - F \Upsilon^{-1}\qu{A_{\mathrm{LVE}_{\phi}}^2+\p{A_0\odot \expodot \Id}A_{\mathrm{LVE}_\phi}} \\
&=& - F \Upsilon^{-1}\qu{\widehat{A}_{\mathrm{LVE}_{\phi}}A_{\mathrm{LVE}_\phi}} , \\
\ddot{\overbracket[0.5pt]{C}} &=& - F \Upsilon^{-1}\qu{\dot{\overbracket[0.5pt]{\widehat{A}_{\mathrm{LVE}_{\phi}}}}A_{\mathrm{LVE}_\phi}+\widehat{A}_{\mathrm{LVE}_{\phi}}\dot{\overbracket[0.5pt]{A_{\mathrm{LVE}_\phi}}}} \\
&=& - F \Upsilon^{-1}\qu{-A_{\mathrm{LVE}_\phi}\widehat{A}_{\mathrm{LVE}_{\phi}}A_{\mathrm{LVE}_\phi}+
\dot{\overbracket[0.5pt]{\widehat{A}_{\mathrm{LVE}_{\phi}}}}A_{\mathrm{LVE}_\phi}+\widehat{A}_{\mathrm{LVE}_{\phi}}\dot{\overbracket[0.5pt]{A_{\mathrm{LVE}_\phi}}}} \\
&=& - F \Upsilon^{-1}\qu{-A_{\mathrm{LVE}_\phi}\widehat{A}_{\mathrm{LVE}_{\phi}}A_{\mathrm{LVE}_\phi}+
\p{A_{\mathrm{LVE}_\phi}\p{A_0\odot\expodot\Id}-\p{A_0\odot\expodot\Id}A_{\mathrm{LVE}_\phi}}A_{\mathrm{LVE}_\phi}+\widehat{A}_{\mathrm{LVE}_{\phi}}\dot{\overbracket[0.5pt]{A_{\mathrm{LVE}_\phi}}}} \\
&=& - F \Upsilon^{-1}\qu{-A^2_{\mathrm{LVE}_\phi}A_{\mathrm{LVE}_\phi}-\p{A_0\odot\expodot\Id}A^2_{\mathrm{LVE}_\phi}+\widehat{A}_{\mathrm{LVE}_{\phi}}\dot{\overbracket[0.5pt]{A_{\mathrm{LVE}_\phi}}}} \\
&=& F \Upsilon^{-1}\qu{\widehat{A}_{\mathrm{LVE}_{\phi}}A^2_{\mathrm{LVE}_\phi}-\widehat{A}_{\mathrm{LVE}_{\phi}}\dot{\overbracket[0.5pt]{A_{\mathrm{LVE}_\phi}}}}
\end{eqnarray*}}This leads to the following statement, which can be easily proved by induction: for every $i\ge 1$, there exists  such that 
\begin{equation} \label{final}
\frac{d^i C}{dt^i}  = \frac{d^i }{dt^i} \p{F\cdot \Upsilon^{-1} \cdot \widehat{A}_{\mathrm{LVE}_{\phi}}} = F\cdot \Upsilon^{-1} \cdot \widehat{A}_{\mathrm{LVE}_{\phi}} \cdot M_i, \qquad \mbox{for some $M_i\in \Mat\p{K}$.}
\end{equation}
Case $i=0$ is true with $M_0=\expodot \Id$ by definition, cases $i=1$ and $i=2$ are true with $M_1=-A_{\mathrm{LVE}_\phi}$ and $M_2=A^2_{\mathrm{LVE}_\phi}-\dot{\overbracket[0.5pt]{A_{\mathrm{LVE}_\phi}}}$ as seen above and if we assume \eqref{final} true for an arbitary $i$, then the proof for $i+1$ works along the lines of case $i=2$ by way of Lemma \ref{Alem}:
{\small\begin{eqnarray*}
\frac{d^{i+1} C}{dt^{i+1}}  &=& \frac{d}{dt} \p{F \Upsilon^{-1}  \widehat{A}_{\mathrm{LVE}_{\phi}}  M_i} = F \Upsilon^{-1}\qu{-A_{\mathrm{LVE}_\phi}\widehat{A}_{\mathrm{LVE}_{\phi}}  M_i+\dot{\overbracket[0.5pt]{\widehat{A}_{\mathrm{LVE}_{\phi}}}} M_i+ \widehat{A}_{\mathrm{LVE}_{\phi}}\dot{\overbracket[0.5pt]{M_i}}} \\
&=& F \Upsilon^{-1} \qu{-A_{\mathrm{LVE}_\phi}\widehat{A}_{\mathrm{LVE}_{\phi}}  M_i+\p{A_{\mathrm{LVE}_\phi}\p{A_0\odot\expodot\Id}-\p{A_0\odot\expodot\Id}A_{\mathrm{LVE}_\phi}} M_i+ \widehat{A}_{\mathrm{LVE}_{\phi}}\dot{\overbracket[0.5pt]{M_i}}} \\
&=& F \Upsilon^{-1} \qu{-\widehat{A}_{\mathrm{LVE}_{\phi}}A_{\mathrm{LVE}_\phi}  M_i+ \widehat{A}_{\mathrm{LVE}_{\phi}}\dot{\overbracket[0.5pt]{M_i}}} = F \Upsilon^{-1} \widehat{A}_{\mathrm{LVE}_{\phi}}\qu{-A_{\mathrm{LVE}_\phi}  M_i+ \dot{\overbracket[0.5pt]{M_i}}},
\end{eqnarray*}}which explicitly defines the sought-after matrices recursively as 
\[ M_{i+1}=-A_{\mathrm{LVE}_\phi}  M_i+ \dot{\overbracket[0.5pt]{M_i}}, \qquad i\ge 0  . \]
Thus \eqref{final} is true for every $i\ge 1$ and all we need to do to finish the proof of Theorem \ref{filterthm} is ascertain that all derivatives of $C$ at $t=t_0$ are identically equal to the zero infinite matrix. But this is true on account of the fact that if we take \eqref{final} at $t=t_0$, we obtain
\[
\left.\frac{d^{i+1} C}{dt^{i+1}}\right|_{t=t_0} = F \cdot \left.\widehat{A}_{\mathrm{LVE}_{\phi}}\right|_{t=t_0} \cdot M_i \p{t_0} = 0 \cdot M_i\p{t_0} = 0 ,
\]
in virtue of Proposition \ref{prefinal}. Thus the matrix $C=F \Upsilon^{-1}  \widehat{A}_{\mathrm{LVE}_{\phi}}$ of functions holomorphic on the open neighborhood $I:=\phi^{-1}\p{U}\cap V$ of $t_0$ has all of its derivatives with respect to $t$ equal to zero in $t_0$. This means $C\equiv 0$ in a neighborhood $J$ of $t_0$, thus on that neighborhood $\p{\Upsilon^{-1}}^TF^T \in \ker \widehat{A}_{\mathrm{LVE}_{\phi}}^T$ which is the last step we needed to prove the columns of $\p{\Upsilon^{-1}}^TF^T$ are admissible solutions defined on a neighborhood of $\brr{\phi\p{t}:t\in I\cap J}$. Remarks \ref{importantremarks} and Lemma \ref{lemma1} finish the proof. $\square$

\subsubsection{Further comments}

The following corresponds to single-row blocks (i.e. one first integral at a time) and is a valid precursor to future studies on how to maintain the filtered admissible structure even after performing changes of variables, but this is not immediately necessary to our current aims:
\begin{cor} \label{changeofvariables}
Let $g$ be a first integral of \eqref{sys} and $g^{\p{k}}_i$ its $k^{\mathrm{th}}$ lexicographic derivative, for every $k\ge 1$. Let $g_k:=g^{\p{k}}\p{\phi\p{t_0}}$ be its value at the original point and define the matrix (constant, just like $F_k$):
\[ M_k = 
 \p{
\begin{array}{ccccc}
 g_1^{\odot k}  & & & \\
 c_{1,\dots,1,2} g_1^{\odot k-1}\odot g_2 & \ddots & & \\
\vdots & \ddots & g_1^{\odot 2} \\
g_k & \cdots & g_2 & g_1 
  \end{array}
 }
 \in \mathrm{Mat}_{1,n}^{n,n}
\]
be the first $k\times D_{n,k}$ block of $M:=\expodot \p{\cdots g_k \mid g_{k-1}\mid \dots\mid g_2\mid g_1}$.
Then 
\[
\p{\Upsilon_k^{-1}}^T\Phi_k^TM_k^T = \p{\begin{array}{cccc} 
\frac{1}{k!} \p{g^k}^{\p{k}} & \cdots & \frac12 \p{g^2}^{\p{k}}& g^{\p{k}}\\
 & \ddots & \vdots & \vdots \\
 & & \frac{1}{2}\p{g^2}^{\p{2}} & g^{\p{2}} \\
 & & & g^{\p{1}}  \end{array}}
\]
In other words: $\p{\Upsilon^{-1}}^T\Phi^TM^T$ is an infinite matrix whose columns are the Taylor terms of the respective powers of $g$. $\square$
\end{cor}

\begin{example}
We have seen two methods of obtaining every new Taylor block of a given formal first integral: the global filter provided by Theorem \ref{filterthm}, and the progressive quadrature method in Theorem \ref{othermethod}. Let us remain focused on the former and assume, for instance, $i=1$ in Lemma \ref{lemma1}; $F_1$ in \eqref{f0i1} transforms the transpose of the fundamental matrix of the dual system $\Upsilon_1^{-1}=Y_1^{-1}$  into a matrix of the form
\newcommand\ChangeRT[1]{\noalign{\hrule height #1}}
\newcolumntype{?}{!{\vrule width 0.8pt}}
\begin{equation}\label{filtered1}
\p{
\;\begin{array}{|c|}\hline
\bm{0} \\ \hline
\rule{1cm}{0cm} G_1 \rule{1cm}{0cm} \\ \hline
\end{array}\;} := \p{
\;\begin{array}{?c?}\ChangeRT{0.8pt}
\bm{0} \\ \ChangeRT{0.8pt}
\rule{1cm}{0cm} g_1^{\p{1}} \rule{1cm}{0cm} \\ \ChangeRT{0.1pt}
 g_2^{\p{1}} \\ \ChangeRT{0.1pt}
\vdots \\ \ChangeRT{0.1pt}
g_{n-1}^{\p{1}} \\\ChangeRT{0.8pt}
\end{array}\;}
=
\p{\begin{array}{ccccc}
0 & 0 & \cdots & 0 & 0 \\
g_{1,1}^{\p{1}} & g_{1,2}^{\p{1}} & \cdots & g_{1,n-1}^{\p{1}}  & g_{1,n}^{\p{1}}  \\
g_{2,1}^{\p{1}} & g_{2,2}^{\p{1}} & \cdots & g_{2,n-1}^{\p{1}}  & g_{2,n}^{\p{1}}  \\
\vdots & \vdots &  & \vdots  & \vdots  \\
g_{n-1,1}^{\p{1}} & g_{n-1,2}^{\p{1}} & \cdots & g_{n-1,n-1}^{\p{1}}  & g_{n-1,n}^{\p{1}}  \\
\end{array}};
\end{equation}
each of these row vectors $g_1^{\p{1}} ,\dots,g_{n-1}^{\p{1}}$ is the gradient of a formal first integral and all $n-1$ of them are linearly independent. The next step would be to integrate the second variational system (from here onwards only quadratures are needed):
\[
Y_2\p{t} = Y_1\int_0^t Y_1^{-1} \p{\tau} A_2\p{\tau} Y_1^{\odot 2} \p{\tau} d \tau
\]
invert the fundamental matrix for $\p{\mathrm{LVE}^2}$, 
\[
\Upsilon_2 = \p{\begin{array}{cc} Y_1^{\odot 2} & 0 \\ Y_2 & Y_1 \end{array}} , \qquad \mbox{thus} \qquad \Psi_2:=\p{\Upsilon_2^{-1}}^T
\]
and multiply it by the second-order filter matrix
\[
\Phi_2 = \p{\begin{array}{cc} F_1^{\odot 2} & 0 \\ F_2 & F_1 \end{array}},
\]
where $F_2$ is obtained from \eqref{Fkstatement1}, \eqref{Fkstatement2},
and the result we obtain is 
\begin{equation} \label{phi2}
\Phi_2 \p{\Psi_2}^T =\Phi_2\Upsilon_2^{-1}= \p{\begin{array}{cc} G_1^{\odot 2} & 0 \\ G_2 & G_1 \end{array}} = 
\p{\begin{array}{cc}
{\tiny 
\begin{array}{|c|}\ChangeRT{0.5pt}
\bm{0}_n \\ \ChangeRT{0.1pt}
  \rowcolor{Lightgray!50!white}\rule{0.45cm}{0cm} g_1^{\p{1}} \rule{0.45cm}{0cm} \\ \ChangeRT{0.1pt}
  \rowcolor{white}\vdots \\ \ChangeRT{0.1pt}
g_{n-1}^{\p{1}} \\ \ChangeRT{0.5pt}
\end{array}\;}^{\odot 2} & \\ \ChangeRT{1.1pt}
\multicolumn{1}{|c|}{\bm{0}_{d_{n,2}}} & \multicolumn{1}{c|}{\bm{0}_n} \\ \ChangeRT{0.1pt}
  \rowcolor{Lightgray}\multicolumn{1}{|c|}{\rule{2cm}{0cm}g_1^{\p{2}}\rule{2cm}{0cm}} & \multicolumn{1}{c|}{\rule{0.8cm}{0cm} g_1^{\p{1}} \rule{0.8cm}{0cm}} \\ \ChangeRT{0.1pt}
\multicolumn{1}{|c|}{\vdots } & \multicolumn{1}{c|}{\vdots} \\ \ChangeRT{0.1pt}
\multicolumn{1}{|c|}{g_{n-1}^{\p{2}}} & \multicolumn{1}{c|}{g_{n-1}^{\p{1}}} \\ \ChangeRT{1.1pt}
\end{array}}
\end{equation}
Take for instance the first non-zero row of the lower block $\p{G_2\mid G_1}$, i.e. the one highlighted in gray. Imagine we perform the computations for $\p{\mathrm{LVE}_\phi^3}$ and multiply $\Upsilon_3^{-1}$ by the filter matrix $\Phi_3$; we obtain a matrix having the one in \eqref{phi2} as a lower right $D_{n,k}\times D_{n,k}$ block:
{\small\begin{eqnarray} \label{g1}
\Phi_3 \Upsilon_3^{-1} \!&\!=\!& \!\p{\begin{array}{ccc} F_1^{\odot 3} & 0 & 0 \\ 3 F_1 \odot F_2  & F_1^{\odot} & 0 \\ F_3 & F_2 & F_1 \end{array}}\Upsilon_3^{-1}=\p{\begin{array}{ccc} G_1^{\odot 3} & 0 & 0 \\ 3 G_1 \odot G_2  & G_1^{\odot} & 0 \\ G_3 & G_2 & G_1 \end{array}} \\
&=& \nonumber
\p{\begin{array}{ccc}
{\tiny 
\begin{array}{|c|}\ChangeRT{0.5pt}
\bm{0}_n \\ \ChangeRT{0.1pt}
  \rowcolor{Lightgray!50!white}\rule{0.45cm}{0cm} g_1^{\p{1}} \rule{0.45cm}{0cm} \\ \ChangeRT{0.1pt}
  \rowcolor{white}\vdots \\ \ChangeRT{0.1pt}
g_{n-1}^{\p{1}} \\ \ChangeRT{0.5pt}
\end{array}\;}^{\odot 3} & & \\
3\; {\tiny 
\begin{array}{|c|}\ChangeRT{0.5pt}
\bm{0}_n \\ \ChangeRT{0.1pt}
  \rowcolor{Lightgray!50!white}\rule{0.45cm}{0cm} g_1^{\p{1}} \rule{0.45cm}{0cm} \\ \ChangeRT{0.1pt}
  \rowcolor{white}\vdots \\ \ChangeRT{0.1pt}
g_{n-1}^{\p{1}} \\ \ChangeRT{0.5pt}
\end{array}\;} \odot {\tiny \;
\begin{array}{|c|}\ChangeRT{0.5pt}
\bm{0}_{d_{n,2}}  \\ \ChangeRT{0.1pt}
  \rowcolor{Lightgray!50!white}\rule{1cm}{0cm}g_1^{\p{2}}\rule{1cm}{0cm}  \\ \ChangeRT{0.1pt}
\rowcolor{white}\vdots   \\ \ChangeRT{0.1pt}
g_{n-1}^{\p{2}}  \\ \ChangeRT{0.5pt}
\end{array}} & {\tiny 
\begin{array}{|c|}\ChangeRT{0.5pt}
\bm{0}_n \\ \ChangeRT{0.1pt}
  \rowcolor{Lightgray!50!white}\rule{0.45cm}{0cm} g_1^{\p{1}} \rule{0.45cm}{0cm} \\ \ChangeRT{0.1pt}
  \rowcolor{white}\vdots \\ \ChangeRT{0.1pt}
g_{n-1}^{\p{1}} \\ \ChangeRT{0.5pt}
\end{array}\;}^{\odot 2}  & \\ \ChangeRT{1.1pt}
\multicolumn{1}{|c|}{\bm{0}_{d_{n,3}}}& \multicolumn{1}{c|}{\bm{0}_{d_{n,2}}} & \multicolumn{1}{c|}{\bm{0}_n} \\ \ChangeRT{0.1pt}
\rowcolor{Lightgray} \multicolumn{1}{|c}{\rule{2cm}{0cm}g_1^{\p{3}}\rule{2cm}{0cm}} &  \multicolumn{1}{|c|}{\rule{2cm}{0cm}g_1^{\p{2}}\rule{2cm}{0cm}} & \multicolumn{1}{c|}{\rule{0.8cm}{0cm} g_1^{\p{1}} \rule{0.8cm}{0cm}} \\ \ChangeRT{0.1pt}
\multicolumn{1}{|c}{\vdots } & \multicolumn{1}{|c|}{\vdots } & \multicolumn{1}{c|}{\vdots} \\ \ChangeRT{0.1pt}
\multicolumn{1}{|c}{g_{n-1}^{\p{n-1}}} & \multicolumn{1}{|c|}{g_{n-1}^{\p{2}}} & \multicolumn{1}{c|}{g_{n-1}^{\p{1}}} \\ \ChangeRT{1.1pt}
\end{array}}\label{phi3}
\end{eqnarray}}thus we have an increasing row, highlighted in darker gray color $\dots \mid g_1^{\p{3}} \mid g_1^{\p{2}} \mid g_1^{\p{1}}$. These vectors are precisely the terms of a formal first integral $g_1\p{\bm z}$ of \eqref{sys} along solution $\phi$:
\begin{equation} \label{g1eq}
g_1\p{\bm z+\phi } = g_1\p{\phi} + g_1^{\p{1}} \bm z+\frac{1}{2}g_1^{\p{2}}\p{\phi} \bm z^{\odot 2}+\frac{1}{6}g_1^{\p{3}}\p{\phi} \bm z^{\odot 3}+\dots
\end{equation}
where $ g_1^{\p{i}}= g_1^{\p{i}}\p{\bm \phi}$, i.e. the implementation of \eqref{lexdef} in \ref{notalex} (\ref{notalex3}) to $\bm z = \phi$. The same applies, \emph{mutatis mutandis}, to the rest of rows of infinite block $\dots G_3\mid G_2 \mid G_1$, i.e. replacing $g^{\p{i}}_1$ with $g^{\p{i}}_j$ for some other $j=2,\dots,n-1$, which will lead to another first integral $g_j$ in lieu of $g_1$ in \eqref{g1eq}. Needless to say the upper terms $3G_1\odot G_2$, etcetera in \eqref{g1} and higher-order counterparts contain all the cross-products appearing in derivatives of products of the first integrals, e.g. $g_1^{\p{1}}\odot g_3^{\p{2}}$ would be a row in $G_1\odot G_2$ (it is easy to check which one by way of lexicographic ordering and we invite the reader to do this themselves, using the examples in \S \ref{examples}).
\end{example}

\section{Examples} \label{examples}
\subsection{Dixon's system} \label{dixonsection}
The following system,
\begin{equation} \label{CKDeq}
\dot u = \frac{uv}{u^2+v^2} - \alpha u , \qquad \dot v = \frac{v^2}{u^2+v^2} - \beta v +\beta -1, 
\end{equation}
where $\alpha,\beta>0$, arose in \cite{CDK} as a decoupled two-dimensional fragment of a three-dimensional dynamical model of the magnetic field of neutron stars. Several references (e.g. \cite{ADE,DCK,Sprott}) discussed the ostensible chaotic behavior of the dynamics of \eqref{CKDeq}. \cite{SeilerSeiss}, on the other hand, purported the non-chaotic character of this system based on a version of the Poincar\'e-Bendixson Theorem that precludes the need for compact sets comprising finitely many sets of homoclinic/heteroclinic connections. 
Given the existing, seldom-discussed gap between the concepts \emph{non-chaotic} and \emph{integrable} (regardless of which ad-hoc definition of chaos is used), we can afford to eschew the polemic altogether and focus on finding a conserved quantity for \eqref{CKDeq} using higher variational equations.

We will focus on the case $\alpha=\beta>0$. The case $\alpha\ne \beta$, as well as the original three-dimensional model (\cite[eqs (3)-(5)]{DCK}), will be tackled in future work.

$\alpha=1$ is immediate to solve and needs no further consideration. Let us focus on $\alpha \ne 1$.
Previous attempts at simplifying higher-order $\mathrm{LVE}^k$ suggest change of variable 
\begin{equation} \label{changeCDK}
\p{u,v} = \p{\frac{\sin y}{x},\frac{\cos y}{x}}
\end{equation}
which transforms \eqref{CKDeq} for $\alpha=\beta$ into 
\begin{equation} \label{CKDeqtransf}
\dot x = \alpha x \p{1-x \cos y},\qquad \dot y = -\p{\alpha-1} x \sin y .
\end{equation}
We will use particular solution 
\begin{equation} \label{parts1}
\phi\p{t} = \p{x\p{t},0} \qquad \mbox{ such that }\dot x = -\alpha \p{-1+x}x.
\end{equation}
Along $\Gamma$,  the higher variational complex \label{ALVE} and its dual \eqref{LVEstar} feature the following blocks for $k\ge 0$, 
\begin{equation} \label{generalA}
A_k= \left\{ \begin{array}{ll} 
X\p{x,y}=\p{ \begin{array}{c}\alpha x \p{1-x \cos y}\\-\p{\alpha-1} x \sin y\end{array}} , & k=0, \\
\left(
\begin{array}{cc}
 -\alpha \p{2 x-1} & 0 \\
 0 & -\p{\alpha-1} x \\
\end{array}
\right) , & k=1, \\
\p{\begin{array}{cccccc}
\cline{1-3}
\multicolumn{3}{|c|}{\multirow{2}{*}{$0_{2\times \p{k-2}}$}} &  \p{-1}^{k/2} 2 \alpha & 0 & \p{-1}^{k/2+1}\alpha x^2 \\
\multicolumn{3}{|c|}{} & 0 & \p{-1}^{k/2}\p{-1+\alpha} & 0 \\
\cline{1-3}
\end{array}}, & k\mbox{ even} \\
& \\ 
\p{\begin{array}{ccccc}
\cline{1-3}
\multicolumn{3}{|c|}{\multirow{2}{*}{$0_{2\times \p{k-1}}$}} &  \p{-1}^{\frac{k+1}{2}} 2 \alpha x & 0  \\
\multicolumn{3}{|c|}{} & 0 & \p{-1}^{\frac{k+1}{2}}\p{-1+\alpha} x \\
\cline{1-3}
\end{array}}, & k \mbox{ odd},
\end{array}
\right.
\end{equation}
for instance a principal fundamental matrix of \eqref{VE} is
\[
\Upsilon_1=Y_1=\left(
\begin{array}{cc}
 \frac{(x-1) x}{(x_0-1) x_0} & 0 \\
 0 & (1-x_0)^{\frac{1}{\alpha}-1} (1-x)^{1-\frac{1}{\alpha}} \\
\end{array}
\right)
\]
and the first-order filter matrix is very simple in this case,
\[
F_1=\Id_2-\frac{1}{X_1^0}\p{X^0\odot \bm{e}_1^T} = \left(
\begin{array}{cc}
 0 & 0 \\
 -\frac{X_2\p{x_0,y_0}}{X_1\p{x_0,y_0}} & 1 \\
\end{array}
\right)=\left(
\begin{array}{cc}
 0 & 0 \\
0 & 1 \\
\end{array}
\right),
\]
where $x_0=x\p{0}$ and $y_0=y\p{0}$; so are the higher-order matrices in $\Phi = \expodot F$ obtained from \eqref{Fkstatement1}, some of which are shown below:
\begin{eqnarray*}
&F_2=\left(
\begin{array}{ccc}
 0 & 0 & 0 \\
 0 & \frac{1-\alpha}{\alpha (x_0-1)} & 0 \\
\end{array}
\right), \quad F_3 = \left(
\begin{array}{cccc}
 0 & 0 & 0 & 0 \\
 0 & \frac{(\alpha-1) (2 \alpha-1)}{\alpha^2 (x_0-1)^2} & 0 & 0 \\
\end{array}
\right),& \\
& F_4 = \left(
\begin{array}{ccccc}
 0 & 0 & 0 & 0 & 0 \\
 0 & -\frac{(\alpha-1) (2 \alpha-1) (3 \alpha-1)}{\alpha^3 (x_0-1)^3} & 0 & -\frac{(\alpha-1) (2 x_0+1)}{\alpha (x_0-1)^2} & 0 \\
\end{array}
\right),\dots&
\end{eqnarray*}
Returning to $k=1$, Lemma \ref{lemma1} ensures that the second column of the filtered matrix
\[
\p{Y_1^{-1}}^T F_1^T = \left(
\begin{array}{cc}
 0 & 0 \\
 0 & (1-x_0)^{1-\frac{1}{\alpha}} (1-x)^{\frac{1}{\alpha}-1} \\
\end{array}
\right),
\]
is an admissible solution of degree $1$. We can normalize this to eliminate appearances of $x_0$:
\[
\p{\Phi_1^{-1}}^T F_1^T P^T := \left(
\begin{array}{cc}
 0 & 0 \\
 0 & (1-x_0)^{1-\frac{1}{\alpha}} (1-x)^{\frac{1}{\alpha}-1} \\
\end{array}
\right) \left(
\begin{array}{cc}
 1 & 0 \\
 0 & (1-x_0)^{\frac{1}{\alpha}-1} \\
\end{array}
\right)^T = \left(
\begin{array}{c|c|}
\cline{2-2}
 0 & 0 \\
 0 & (1-x)^{\frac{1}{\alpha}-1} \\
 \cline{2-2}
\end{array}
\right),
\]
the second of whose columns is still the gradient $f^{\p{1}}$ of a formal first integral (this is exactly what appears as a row $g_1^{\p{1}}$ in \eqref{filtered1}). 

The higher order trimming (in which we discard redundant symmetric products $f^{\p{i_1}}\odots f^{\p{i_k}}$, i.e. derivatives of powers of $f$) is clear once we write down the dual system of equations and filter out telescopically using either Lemma \ref{filteringout1} and the general form of that each new term $f^{\p{k}}$ of our formal first integral $f$ is expressible as the following (bear in mind that $x=x\p{t}$ here): $f^{\p{k}} \p{x} = \p{ \frac{d}{d x} f^{\p{k-1}}  \mid g_k\p{x}}$, where
\[
\sum _{j=0}^{\frac{k-1}{2}}\p{\p{-1}^{j+1}\p{\alpha-1} \binom{k}{ -1 - 2 j + k} g_{k-2j}\p{x}+\p{-1}^{j+1} \alpha \binom{k}{ -2 j + k} x g_{k-2j}'\p{x}}+\alpha g_k'\p{x} =0,
\]
which means $g_k\p{x}=0$ for every even $k\ge 2$ and e.g. 
\begin{eqnarray*}
f_1&=&\left(
\begin{array}{cc}
 0 & (1-x)^{\frac{1}{\alpha}-1} \\
\end{array}
\right), \\ 
f_2&=&\left(
\begin{array}{ccc}
 0 & \frac{(\alpha-1) (1-x)^{\frac{1}{\alpha}-2}}{\alpha} & 0 \\
\end{array}
\right), \\ 
f_3 &=& \left(
\begin{array}{cccc}
 0 & \left(\frac{1}{\alpha}-2\right) \left(\frac{1}{\alpha}-1\right) (1-x)^{\frac{1}{\alpha}-3} & 0 & \frac{(1-x)^{\frac{1}{\alpha}-2} (2 \alpha+(\alpha-2) x-1)}{\alpha-2} \\
\end{array}
\right), \\
f_4 &=& \p{\begin{array}{c|c} \frac{d}{dx} f_3 & 0 \end{array}}, \\
f_5 &=& \p{\begin{array}{c|c}
\frac{d}{dx} f_4 & -\frac{(1-x)^{\frac{1}{\alpha}-3} (-(\alpha-2) x ((\alpha-2) (3 \alpha-4) x+\alpha (39 \alpha-55)+14)+\alpha (2 (53-24 \alpha) \alpha-73)+16)}{(\alpha-2)^2 (3 \alpha-4)} \end{array}}, \\
\end{eqnarray*}
This can be achieved via \eqref{newrecursion} in Theorem \ref{othermethod} (the easiest option even for low values of $k$) or the filter matrix described in \S \ref{filtersection} to compute each $f_k$ in terms of the preceding $f_{k-1},\dots,f_1$. We invite the reader to apply either method and realize that the above terms appear with only minor attention paid at each step.

A tedious scrutiny of the general form of these vectors shows that they are precisely limit cases of the higher-order Taylor terms (along $\p{x\p{t},0}\in \Gamma$) of the following function:
\[
f = \left(x \sin ^{\frac{\alpha}{1-\alpha}}(y)-\p{\frac{\alpha \cos y}{\alpha-1}} {}_2F_1\left(\frac{1}{2},\frac{\alpha}{2 (\alpha-1)}+1;\frac{3}{2};\cos^2y\right)\right){}^{\frac{1}{\alpha}-1}
\]
where $_2F_1$ is the \textbf{Gaussian hypergeometric function} \cite{Abramowitz}.
Removing the outer power yields a first integral of \eqref{CKDeqtransf} as well, as can be easily checked: 
\[
f = x \sin ^{\frac{\alpha}{1-\alpha}}(y)-\p{\frac{\alpha \cos y}{\alpha-1}} {}_2F_1\left(\frac{1}{2},\frac{\alpha}{2 (\alpha-1)}+1;\frac{3}{2};\cos^2y\right). \]
Undoing the transformation \eqref{changeCDK} we obtain:

\begin{thm}
The following function 
\begin{equation} \label{fidix}
\fbox{\begin{tabular}{c}
$F = u^{-\frac{\alpha}{\alpha-1}} \left(u^2+v^2\right)^{\frac{1}{2 (\alpha-1)}}
-\frac{\alpha v \, _2F_1\left(\frac{1}{2},\frac{\alpha}{2 (\alpha-1)}+1;\frac{3}{2};\frac{v^2}{u^2+v^2}\right)}{(\alpha-1) \sqrt{u^2+v^2}}$
\end{tabular} }
\end{equation}
is a first integral of the original special CDK system. $\square$
\end{thm}

\subsection{The SIR model with vital dynamics}

The \textbf{SIR epidemiological model with vital dynamics} \cite{DaleyGani,KermackMcKendrick} is given by 
\begin{equation} \label{SIR}
\left\{ \begin{array}{lcl} 
\dot{S} & = & \mu (n-S)-\frac{\beta S I}{n},\\
\dot{I} & = & \frac{\beta S I}{n}-I (\gamma+\mu), \\
\dot{R} &=& \gamma I-\mu R.
\end{array} \right. 
\end{equation}

The system was first introduced in \cite{KermackMcKendrick} strictly for the modelling of infectious diseases, but has since been applied to 
a number of situations allowing dynamic compartmentalization, e.g. marketing \cite{RS}.
The study of its integrability has been the subject of several works already (\cite{Chen,HarkoLoboMak}) and the search for a first integral is still an open problem. We will address specific examples; the general case will be addressed in the future.

First, let us start with a very simple example whose solution we already know: $\beta=\gamma$ and $\mu=0$ and change of variable
\begin{equation}\label{transSIR1}
S=n \p{1 + x y}, \quad I = -n y, \quad R = \gamma n z,
\end{equation}
we obtain system of equations $\frac{d}{dt}\p{x,y,z}=\p{\gamma-\gamma (x-1) x y,\gamma x y^2,-y}$, and applying Theorem \ref{filterthm}and Corollary \ref{changeofvariables} and (if the matrix is used) Lemma \ref{filteringout1}, the general terms of the jets telegraph a very identifiable structure: that of two first integrals for the transformed system, $g^\star = \frac{\log (x y+1)}{\gamma}+z$, $f^\star = -x y+\log (x y+1)+y$ which transform back by \eqref{transSIR1} into $g=\frac{n \log \left(\frac{S}{n}\right)+R}{\gamma n}$, $f=-\frac{S+I}{n}+\log \left(\frac{S}{n}\right)+1$, thus taking us to the already-known first-integral $S+ I + R-n$ (thus, $S+ I + R$) and the remaining one, $\frac{n \log \left(\frac{S}{n}\right)+R}{\gamma n}$ (alternatively, $\frac{S }{n}e^{\frac{R}{n}}$). This  is all in keeping with what is already known about the integrability of the system for $\mu=0$ (see e.g. \cite[\S 1]{Chen}). The reader can fill in the blanks for filter matrices or progressive filtering (Theorem \ref{othermethod}) and formal Taylor terms here; we will add more detail about these in the less trivial case below.

To wit, the case in which $\gamma=0$ and $\beta,\mu\ne 0$. To the best of our knowledge, the integrability of this case is an open problem in earlier references. Change of variables
\begin{equation}\label{newchange}
\p{S,I,R} = \p{y^2 \left(x^{-\frac{\beta}{\mu}}+1\right)+n,-x^{-\frac{\beta}{\mu}} y^2 ,y^2 z},
\end{equation}
itself suggested by an earlier try at simplifying $\mathrm{LVE}^3_\phi$, provides us with new system
\begin{equation} \label{newSIR}
\dot x  =  -\frac{\mu x \left[y^2 \left(x^{-\frac{\beta}{\mu}}+1\right)+n\right]}{n}, \quad
\dot y = -\frac{1}{2} \mu y, \quad
\dot z = 0.
\end{equation}
This system admits an invariant plane $y=z=0$ containing particular solution 
\begin{equation} \label{parts}
\phi\p{t}=\p{\widehat{x}\p{t},0,0}, \qquad \mbox{where }\frac{d}{dt}\widehat{x}=-\mu \widehat{x}.
\end{equation} This leads to the variational system $\dot Y =\p{A\odot \expodot \Id_n} Y$ along $\phi$ where
\[
A_{\mathrm{LVE}_{\phi}} = A\odot \expodot \Id_n  = \p{
\begin{array}{ccccc}
\ddots & & & & \\
\cdots  & 4 A_1\odot \Id_n^{\odot 3}  & & & \\
\cdots  & 6 A_2\odot \Id_n^{\odot 2} & 3 A_1\odot \Id_n^{\odot 2} & & \\
\cdots & 4 A_3 \odot \Id_n & 3 A_2 \odot \Id_n & 2 A_1 \odot \Id_n & \\
\cdots & A_4 & A_3 & A_2 & A_1 
 \end{array}
 },
\]
and $A_i$ defined per \eqref{jetA} appear as follows:
\begin{eqnarray*}
A_0 & = & \left(
\begin{array}{c}
 -\mu \widehat{x}(t) \\
 0 \\
 0 \\
\end{array}
\right)\\
A_1 & = & \left(
\begin{array}{ccc}
 -\mu & 0 & 0 \\
 0 & -\frac{\mu}{2} & 0 \\
 0 & 0 & 0 \\
\end{array}
\right)\\
A_2 & = & \left(
\begin{array}{cccccc}
 0 & 0 & 0 & \frac{2 \mu \widehat{x}(t) \left(-\widehat{x}(t)^{-\frac{\beta}{\mu}}-1\right)}{n} & 0 & 0 \\
 0 & 0 & 0 & 0 & 0 & 0 \\
 0 & 0 & 0 & 0 & 0 & 0 \\
\end{array}
\right)\\
A_3 & = & \left(
\begin{array}{cccccccccc}
 0 & 0 & 0 & \frac{2 \left((\beta-\mu) \widehat{x}(t)^{-\frac{\beta}{\mu}}-\mu\right)}{n} & 0 & 0 & 0 & 0 & 0 & 0 \\
 0 & 0 & 0 & 0 & 0 & 0 & 0 & 0 & 0 & 0 \\
 0 & 0 & 0 & 0 & 0 & 0 & 0 & 0 & 0 & 0 \\
\end{array}
\right)\\
A_k & = & \left(
\begin{array}{ccccccc}
\cline{5-7} 0 & 0 & 0 & \frac{2 \left((\beta-\mu) \widehat{x}(t)^{-\frac{\beta}{\mu}}-\mu\right)}{n} & \multicolumn{3}{|c|}{\multirow{3}{*}{$0_{3\times d_{3,k-4}}$}}   \\
 0 & 0 & 0 & 0 &\multicolumn{3}{|c|}{\rule{1cm}{0cm}}   \\
 0 & 0 & 0 & 0 &\multicolumn{3}{|c|}{}\\
 \cline{5-7}
\end{array}
\right), \qquad k\ge 4
\end{eqnarray*}
Lower blocks of the quadrature-free filter matrix are built as proposed in \S \ref{filtersection}, let us write the first few cases
{\small\[
F_1\!=\!\left(
\begin{array}{ccc}
 0 & 0 & 0 \\
 0 & 1 & 0 \\
 0 & 0 & 1 \\
\end{array}
\right), \,
F_2 \!=\! \left(\!
\begin{array}{ccc}\cline{3-3}
 0 & 0 & \multicolumn{1}{|c|}{\multirow{3}{*}{\!$0_{3\times 4}$\!}} \\
 0 & -\frac{1}{2 \widehat{x}_0} & \multicolumn{1}{|c|}{\rule{0.cm}{0cm}} \\
 0 & 0 & \multicolumn{1}{|c|}{\rule{0.cm}{0cm}} \\ \cline{3-3}
\end{array}
\right), \, F_3 \! =\!  \left(
\begin{array}{ccc}\cline{3-3}
 0 & 0 & \multicolumn{1}{|c|}{\multirow{3}{*}{\!$0_{3\times 8}$\!}} \\
 0 & \frac{3}{4 \widehat{x}_0^2} & \multicolumn{1}{|c|}{\rule{0.5cm}{0cm}} \\
 0 & 0 & \multicolumn{1}{|c|}{} \\ \cline{3-3}
\end{array}
\right), \,
F_4 \!=\!
\left(
\begin{array}{ccccc}\cline{3-3} \cline{5-5}
 0 & 0 & \multicolumn{1}{|c|}{\multirow{3}{*}{\!$0_{3\times 4}$\!}} & 0 & \multicolumn{1}{|c|}{\multirow{3}{*}{\!$0_{3\times 8}$\!}} \\
 0 & -\frac{15}{8 \widehat{x}_0^3} & \multicolumn{1}{|c|}{\rule{0.5cm}{0cm}} & \frac{3 \left(\widehat{x}_0^{-\frac{\beta}{\mu}}+1\right)}{n \widehat{x}_0} & \multicolumn{1}{|c|}{\rule{0.5cm}{0cm}} \\
 0 & 0 & \multicolumn{1}{|c|}{\rule{0.5cm}{0cm}} & 0 & \multicolumn{1}{|c|}{} \\ \cline{3-3} \cline{5-5}
\end{array}
\right).\]}
The filtered matrix, has two linearly independent admissible solutions as its nonzero rows:
\[
F_1Y_1^{-1} = \left(
\begin{array}{ccc}
 0 & 0 & 0 \\
0 & \frac{\sqrt{\widehat{x}_0}}{\sqrt{\widehat{x}(t)}} & 0 \\
0 & 0 & 1 \\
\end{array}
\right),
\]
which (optionally) multiplied by
\begin{equation}\label{Pint}
P_1=\left(
\begin{array}{ccc}
 1 & 0 & 0 \\
 0 & \frac{1}{\sqrt{\widehat{x}_0}} & 0 \\
 0 & 0 & 1 \\
\end{array}
\right)
\end{equation}
removes occurrences of the initial condition $\widehat{x}_0$:
\[
G_1 = 
\left(
\begin{array}{ccc}
 0 & 0 & 0 \\
\cline{1-3}\multicolumn{1}{|c}{0} & \frac{1}{\sqrt{\widehat{x}(t)}} & \multicolumn{1}{c|}{0} \\ \cline{1-3}
\multicolumn{1}{|c}{0} & 0 & \multicolumn{1}{c|}{1} \\
\cline{1-3}
\end{array}
\right) = \left(
\begin{array}{ccc}
 0 & 0 & 0 \\ \cline{1-3}
\multicolumn{3}{|c|}{f^{\p{1}}} \\ \cline{1-3}
\multicolumn{3}{|c|}{g^{\p{1}}} \\ \cline{1-3}
\end{array}
\right).
\]
Let us write things down for $k=3$: once we have solved the direct variational system $\p{\mathrm{LVE}_\Gamma^3}$ and obtained its fundamental matrix $\Upsilon_3$, the principal fundamental matrix of the dual system is nothing but the transpose of its inverse; we can either keep said transpose $\p{\Upsilon_3^{-1}}^T$ and right-multiply it by proposed filter $\Phi_3^T$ (in which case first integral jets appear as columns) or left-multiply the original $\Upsilon_3^{-1}$ by $\Phi_3$ in which case jets appear as rows in the bottom block:
{\small\begin{eqnarray*}
\Phi_3\Upsilon_3^{-1} &=& \left(
\begin{array}{ccc}
 F_1^{\odot 3} & 0 & 0 \\
 3F_1\odot F_2 & F_1^{\odot 2} & 0 \\
 F_3 & F_2 & F_1 \\
\end{array}
\right)\left(
\begin{array}{ccc}
 Y_1^{\odot 3} & 0 & 0 \\
 3Y_1\odot Y_2 & Y_1^{\odot 2} & 0 \\
 Y_3 & Y_2 & Y_1 \\
\end{array}
\right)^{-1} \\
&=& \left(
\begin{array}{ccccccccccccccccccc} \hline
\multicolumn{19}{|c|}{\multirow{2}{*}{$ 0_{6\times 19} $}}  \\
\multicolumn{19}{|c|}{} \\ \hline
 0 & 0 & 0 & 0 & 0 & 0 & \frac{\widehat{x}_0^{3/2}}{\widehat{x}^{3/2}} & 0 & 0 & 0 & 0 & 0 & 0 & 0 & 0 & 0 & 0 & 0 & 0 \\
 0 & 0 & 0 & 0 & 0 & 0 & 0 & \frac{\widehat{x}_0}{\widehat{x}} & 0 & 0 & 0 & 0 & 0 & 0 & 0 & 0 & 0 & 0 & 0 \\
 0 & 0 & 0 & 0 & 0 & 0 & 0 & 0 & \frac{\sqrt{\widehat{x}_0}}{\sqrt{\widehat{x}}} & 0 & 0 & 0 & 0 & 0 & 0 & 0 & 0 & 0 & 0 \\
 0 & 0 & 0 & 0 & 0 & 0 & 0 & 0 & 0 & 1 & 0 & 0 & 0 & 0 & 0 & 0 & 0 & 0 & 0 \\
 0 & 0 & 0 & 0 & 0 & 0 & 0 & 0 & 0 & 0 & 0 & 0 & 0 & 0 & 0 & 0 & 0 & 0 & 0 \\
 0 & 0 & 0 & 0 & 0 & 0 & 0 & 0 & 0 & 0 & 0 & 0 & 0 & 0 & 0 & 0 & 0 & 0 & 0 \\
 0 & 0 & 0 & 0 & 0 & 0 & 0 & 0 & 0 & 0 & 0 & 0 & 0 & 0 & 0 & 0 & 0 & 0 & 0 \\
 0 & 0 & 0 & -\frac{\widehat{x}_0}{\widehat{x}^2} & 0 & 0 & 0 & 0 & 0 & 0 & 0 & 0 & 0 & \frac{\widehat{x}_0}{\widehat{x}} & 0 & 0 & 0 & 0 & 0 \\
 0 & 0 & 0 & 0 & -\frac{\sqrt{\widehat{x}_0}}{2 \widehat{x}^{3/2}} & 0 & 0 & 0 & 0 & 0 & 0 & 0 & 0 & 0 & \frac{\sqrt{\widehat{x}_0}}{\sqrt{\widehat{x}}} & 0 & 0 & 0 & 0 \\
 0 & 0 & 0 & 0 & 0 & 0 & 0 & 0 & 0 & 0 & 0 & 0 & 0 & 0 & 0 & 1 & 0 & 0 & 0 \\
 0 & 0 & 0 & 0 & 0 & 0 & 0 & 0 & 0 & 0 & 0 & 0 & 0 & 0 & 0 & 0 & 0 & 0 & 0 \\
 0 & \frac{3 \sqrt{\widehat{x}_0}}{4 \widehat{x}^{5/2}} & 0 & 0 & 0 & 0 & a_{18,7} & 0 & 0 & 0 & 0 & -\frac{\sqrt{\widehat{x}_0}}{2 \widehat{x}^{3/2}} & 0 & 0 & 0 & 0 & 0 & \frac{\sqrt{\widehat{x}_0}}{\sqrt{\widehat{x}}} & 0 \\
 0 & 0 & 0 & 0 & 0 & 0 & 0 & 0 & 0 & 0 & 0 & 0 & 0 & 0 & 0 & 0 & 0 & 0 & 1 \\
\end{array}
\right)
\end{eqnarray*}}
where $\widehat{x}_0=\widehat{x}\p{0}$ and 
\[
a_{18, 7 }=\frac{\widehat{x}(t)^{-\frac{\beta}{\mu}-\frac{3}{2}} \left(3 \mu \widehat{x}(t)^{\beta/\mu} \widehat{x}_0^{\frac{3}{2}-\frac{\beta}{\mu}}+3 (\mu-\beta) \sqrt{\widehat{x}_0} (\widehat{x}_0-\widehat{x}(t)) \widehat{x}(t)^{\beta/\mu}-3 \mu \sqrt{\widehat{x}_0} \widehat{x}(t)\right)}{(\beta-\mu) n}
\]
The bottom two rows are already admissible jets $\star^{\p{3}} \mid \star^{\p{2}} \mid \star^{\p{1}}$, but we can trim them further by multiplying this matrix by  
\[
\Pi_3=\left(
\begin{array}{ccc}
 P_1^{\odot 3} & 0 & 0 \\
 3P_1\odot P_2 & P_1^{\odot 2} & 0 \\
 P_3 & P_2 & P_1 \\
\end{array}
\right)^{T}
\]
where $P_1$ is as in \eqref{Pint} and $P_2,P_3$ are successively chosen so as vanish $\widehat{x}_0$ from the matrix (in this case, all terms in $P_2,P_3$ equal to zero except for $\p{P_3}_{2,7}=\frac{3 \left(\mu \left(-\widehat{x}_0^{-\frac{\beta}{\mu}}-1\right)+\beta\right)}{n \sqrt{\widehat{x}_0} (\beta-\mu)}$). This equates to weeding out copies of successive symmetric products of $f^{\p{i}}$, and/or $g^{\p{j}}$ and harks back to Lemma \ref{filteringout1} and Conjecture \ref{changeofvariables}. The resulting matrix retains its neat echeloned structure
\[
\Pi_3\Phi_3\Upsilon_3^{-1}= \left(
\begin{array}{ccc}
 G_1^{\odot 3} & 0 & 0 \\
 3G_1\odot G_2 & G_1^{\odot 2} & 0 \\
 G_3 & G_2 & G_1 \\
\end{array}
\right)
\]
where lower block $G_3 \mid G_2 \mid G_1$ has independent formal first integral jets as its non-zero rows:
{\begin{equation} \label{jets2}
\left(
\begin{array}{ccccccccccccccccccc}
 0 & 0 & 0 & 0 & 0 & 0 & 0 & 0 & 0 & 0 & 0 & 0 & 0 & 0 & 0 & 0 & 0 & 0 & 0 \\ \hline
 \multicolumn{1}{|c}{0} & \frac{3}{4 \widehat{x}^{5/2}} & 0 & 0 & 0 & 0 & \frac{3 \left(-\mu \widehat{x}^{-\frac{\beta}{\mu}}+\beta-\mu\right)}{(\beta-\mu) n \sqrt{\widehat{x}}} & 0 & 0 & 0 &  \multicolumn{1}{|c}{0} & -\frac{1}{2 \widehat{x}^{3/2}} & 0 & 0 & 0 & 0 &  \multicolumn{1}{|c}{0} & \frac{1}{\sqrt{\widehat{x}}} &  \multicolumn{1}{c|}{0} \\ \hline
  \multicolumn{1}{|c}{0} & 0 & 0 & 0 & 0 & 0 & 0 & 0 & 0 & 0 &  \multicolumn{1}{|c}{0} & 0 & 0 & 0 & 0 & 0 &  \multicolumn{1}{|c}{0} & 0 &  \multicolumn{1}{c|}{1} \\ \hline
\end{array}
\right)
\end{equation}
}in other words, the following are degree-three asymptotic expressions of formal first integrals ($\bm{z}=\p{x,y,z}$ and $\phi=\p{\widehat{x},0,0}$ is the particular solution as usual):
\begin{eqnarray*}
f\p{\bm{z}+\phi} &=& f^{\p{1}}\p{\phi}\bm{z} + \frac{1}{2}f^{\p{2}} \p{\phi} \bm{z}^{\odot 2}+\frac{1}{3!} f^{\p{3}} \bm{z}^{\odot 3}+O\p{\bm{z}^{\odot 4}} \\
&=&
\frac{\large{y}}{\sqrt{\widehat{x}}}-\frac{x \large{y}}{2 \widehat{x}^{3/2}}+\frac{1}{6} \left(\frac{3 \large{y}^3 \left(\frac{\mu \widehat{x}^{1-\frac{\beta}{\mu}}}{\mu-\beta}+\widehat{x}\right)}{n \widehat{x}^{3/2}}+\frac{9 x^2 \large{y}}{4 \widehat{x}^{5/2}}\right)+O\p{\bm{z}^{\odot 4}}, \\
g\p{\bm{z}+\phi} &=& g^{\p{1}}\p{\phi}\bm{z} + \frac{1}{2}g^{\p{2}} \p{\phi} \bm{z}^{\odot 2}+\frac{1}{3!} g^{\p{3}} \bm{z}^{\odot 3}+O\p{\bm{z}^{\odot 4}} = z +O\p{\bm{z}^{\odot 4}} .
\end{eqnarray*}
Again, Theorem \ref{othermethod} does the exact same job faster and with the guarantee that it works, e.g.  
\begin{eqnarray*}
f^{\p{2}} &=& -\qu{\p{Y_1^{-1}}^{\odot 2}}^T \qu{\int_{\widehat{x}_0}^{\widehat{x}} \p{Y_1^{\odot 2}}^Tf^{\p{1}}A_2dx-\p{0,\frac{1}{2 \widehat{x}_0^{3/2}},0,0,0,0}},  \\
f^{\p{3}} &=& -\qu{\p{Y_1^{-1}}^{\odot 3}}^T \qu{\int_{\widehat{x}_0}^{\widehat{x}} \p{Y_1^{\odot 3}}^T\qu{3f^{\p{2}}\p{A_2\odot\Id_3}+f^{\p{1}}A_3}dx-\p{0,-\frac{3}{4 \widehat{x}_0^{5/2}},0,0,0,0,0,0,0,0}}, 
\end{eqnarray*}
and the reader can check that the jet terms are exactly the same as those in \eqref{jets2}.
Either method may sound like a cumbersome process but it only requires special attention in its first stages if savvy simplifications are chosen, as was the case for Dixon's system. Later stages of the bottom row remain zero and accordingly \fbox{$g=z$} is easily checked to be a first integral (formal and convergent) for \eqref{newSIR}; for $f$, and much like in \S \ref{dixonsection}, pattern checks along with simple resummation yield a first integral for system \eqref{newSIR}, 
\[
f = 
\frac{y e^{\frac{y^2}{2 n}}}{\left[e^{\frac{u y^2}{n}} \left(\frac{u y^2}{n}\right)^u \Gamma \left(1-u,\frac{u y^2}{n}\right)+x^u\right]^{\frac{1}{2 u}}}, 
\]
where $u = \frac{\beta}{\mu}$ and $\Gamma\p{\cdot,\cdot}$ is the \textbf{incomplete Gamma function} \cite{Abramowitz};
undoing transformation \eqref{newchange} we obtain something that is now immediate to prove on its own:
\begin{thm}
The following two functions,
\begin{eqnarray}
F &=& \frac{e^{\frac{-n+S+I}{2 n}} \sqrt{-n+S+I}}{\left(\frac{-n+S+I}{I}-\left(\frac{\beta}{\mu n}\right)^{\beta/\mu} e^{\frac{\beta (-n+S+I)}{\mu n}} (-n+S+I)^{\beta/\mu} \Gamma \left(1-\frac{\beta}{\mu},\frac{\beta (-n+S+I)}{\mu n}\right)\right)^{\frac{\mu}{2 \beta}}}, \label{sfir1} \\
G &=& \frac{R}{-n+S+I}. \nonumber
\end{eqnarray}
are two functionally independent first integrals of SIR with $\gamma=0$ \eqref{SIR}. $\square$
\end{thm}
It is worth noting, however, that the only easy goal we can aspire to in most cases will be that of formal first integrals with a compact identifiable general term but without a straightforward convergent expression, and resummation and deep pattern recognition techniques will need to be adapted to each problem. This will also be the subject of future research in the short term; for the time being let us see a short example below.

\subsection{The van der Pol oscillator}

This well-known system is usually presented as an epitome of the proverbial ``simple yet chaotic'' dynamical system \cite{Sprott}
\begin{equation} \label{vanderpol}
\ddot u  = \mu (1 - u^2) \dot u - u 
\end{equation}
Value $\mu=2$ appears to have a significant, yet still unclarified role in potential qualitative simplifications in variable transformations. Thus let us fix this value for the present paper: 
\begin{equation} \label{vanderpol2}
\dot u = v, \qquad
\dot v = 2 (1 - u^2) v - u 
\end{equation}

Perform the change of variables 
\begin{equation} \label{vanderpolchange}
\p{u\p{t},v\p{t}} = \p{ \frac{x\p{t} y\p{t}}{\sqrt{2}},\left(\frac{x\p{t}}{\sqrt{2}}-\frac{1}{\sqrt{2}}\right) y\p{t}}
\end{equation}
that will simplify the corresponding jet, and our transformed system is 
\begin{equation} \label{vanderpol3}
\dot x = -(x-1) x^3 y^2-1,\qquad \dot y = (x-1) x^2 y^3+y.
\end{equation}
The dual infinite variational system $\p{\mathrm{VE}_\Gamma}^\star$ along particular solution 
\begin{equation} \label{theeq}
\Gamma = \brr{\p{x\p{t},0}:t\in \nc}, \qquad \mbox{where} \quad \dot{x}\p{t}=-1,
\end{equation}
has the following very simple matrices:
\begin{eqnarray*}
A_1 &=& \left(
\begin{array}{cc}
 0 & 0 \\
 0 & 1 \\
\end{array}
\right),  \\
A_2 &=& \left(
\begin{array}{ccc}
 0 & 0 & -2 (x(t)-1) x(t)^3 \\
 0 & 0 & 0 \\
\end{array}
\right), \\
A_3 &=& \left(
\begin{array}{cccc}
 0 & 0 & 2 (3-4 x(t)) x(t)^2 & 0 \\
 0 & 0 & 0 & 6 (x(t)-1) x(t)^2 \\
\end{array}
\right), \\
A_4 &=& \left(
\begin{array}{ccccc}
 0 & 0 & 12 (1-2 x(t)) x(t) & 0 & 0 \\
 0 & 0 & 0 & 6 x(t) (3 x(t)-2) & 0 \\
\end{array}
\right), \\
A_5 &=& \left(
\begin{array}{cccccc}
 0 & 0 & 12-48 x(t) & 0 & 0 & 0 \\
 0 & 0 & 0 & 36 x(t)-12 & 0 & 0 \\
\end{array}
\right), \\
A_6 &=& \left(
\begin{array}{ccccccc}
 0 & 0 & -48 & 0 & 0 & 0 & 0 \\
 0 & 0 & 0 & 36 & 0 & 0 & 0 \\
\end{array}
\right), \\
A_k &=& 0_{2\times d_{2,k}}, \quad k\ge 7.
\end{eqnarray*}
The fact that the ODE in \eqref{theeq} is so trivial to solve ($x\p{t}=-t+C$) adds no simplicity to what follows below. The fundamental matrix of $\mathrm{VE}_1$ and the first filter matrix are equally simple:
\[
Y_1 = \left(
\begin{array}{cc}
 1 & 0 \\
 0 & e^{x(0)-x(t)} \\
\end{array}
\right), \qquad F_1 = \left(
\begin{array}{cc}
 0 & 0 \\
 0 & 1 \\
\end{array}
\right)
\]
and the generic form of the admissible solutions after applying Theorem \ref{othermethod} is the following: $f_k=\p{\frac{d}{d x} f_{k-1} \mid g_k\p{x}}$ where $g_k$ is a solution to 
\begin{equation}\label{odevdp}
-(k-1) k (x-1) x^3 g_{k-2}'(x)-g_k'(x)+(k-2) (k-1) k (x-1) x^2 g_{k-2}(x)+k g_k(x) = 0.
\end{equation}
and moreover $g_k\p{x}=0$ if $k$ is even, and $g_1\p{x} = e^x$. Equation \eqref{odevdp} has a solution (which of course only merits being written if $k$ is odd)
\[
g_k \p{x} = e^{k x} \int (k-1) k (\xi-1) \xi^2 e^{-k \xi} \p{(k-2) g_{k-2}\p{\xi}-\xi g'_{k-2}\p{\xi}} d \xi.
\]
The formal first integral of \eqref{vanderpol3} takes the form
\[
F = \sum_{k=1}^\infty \frac{1}{k!} g_k\p{x} y^k = \sum_{i=0}^\infty \p{-1}^i e^{\p{2i+1}x} y^{2i+1} G_{2i+1}; 
\]
the first terms of $G_\star$ are easy to discern, 
\begin{eqnarray*}
G_1 &=& 1, \\ G_3 &=& \int_{x_0}^x e^{-2 \xi} (\xi-1)^2 \xi^2 d\xi, \\
G_5 &=& \int_{x_0}^x\p{\xi-1}^3 e^{-4 \xi} \xi^5d\xi+3 \int_{\tau_0}^{\tau}\p{\xi-1}^2 e^{-2 \xi} \xi^2 \left(\int \p{\eta-1}^2 e^{-2 \eta} \eta^2 d\eta\right)d\xi
\end{eqnarray*}
and the general term is
{\small\[
G_{2i+1} = \sum_{m=1}^i  \sum_{j=1}^{\binom{i-1}{m-1}} D_{i-1,m-1,j}\Phi\p{\p{i-m}\mathbf{C}_{m,i-m-1,j}+\mathbf{1}_m,2\p{i-m}\mathbf{C}_{m,i-m-1,j}+\mathbf{1}_m,3\p{i-m}\mathbf{C}_{m,i-m-1,j}+\mathbf{1}_m}
\]}where we define the following constant vectors: $\mathbf{1}_n:=\p{1,\dots,1}\in \nz^n$, and $C_{\star,\star,\star}$ are the columns of the symmetric product
\[
\Id_k \odot \mathbf{1}^T_{d_{k,l}} = \p{C_{k,l,1} \mid C_{k,l,2} \mid \cdots \mid C_{k,l,d_{k,l+1}}} \in \Mat_{k,d_{k,l+1}}\p{\nz};
\]
(obviously $C_{k,l,j}=0$ for $l<0$), $\Phi$ is defined recursively as follows: for every $\mathbf{a},\mathbf{b},\mathbf{c}\in \nz^m$, 
\[
\Phi\p{\mathbf{a},\mathbf{b},\mathbf{c}} := \left\{ \begin{array}{ll} 
\int_{x_0}^{x} \p{\tau -1}^{a_1+1}e^{-\p{b_1+1}\tau} \tau^{c_1}  d\tau, & m=1, \\
\int_{x_0}^{x} \p{\tau -1}^{a_1+1}e^{-\p{b_1+1}\tau} \tau^{c_1}  \Phi\p{\p{a_2,\dots,a_m},\p{b_2,\dots,b_m},\p{c_2,\dots,c_m}}d\tau, & m>1
\end{array} \right.
\]
and $D_{n,m,j}$ is the $j^{\mathrm{th}}$ term in the vector $\mb{D}_{n,m}$ obtained from discarding terms containing repeating factors from the reverse-order version of $\mb{i}_n^T\stackrel{m}{\odots}\mb{i}_n^T$ and $\mb{i}_n$ is the reverse-order sequence of the first $n$ odd integers $>1$, e.g. 
\[
\mb{D}_{3,2} = \p{3\cdot 5, \; 3\cdot 7, \; 5\cdot 7}=\p{D_{3,2,j}}_{j=1,2,3,4}, \quad \mb{D}_{4,3} = \p{3 \cdot 5 \cdot 7,\; 3 \cdot 5 \cdot 9,\; 3 \cdot 7 \cdot 9,\; 5 \cdot 7 \cdot 9}
\]

\section{Future work and further comments}

Proving that Theorem \ref{filterthm} is still true if we drop the local analyticity assumptions on $X$ and $\Upsilon^{-1}$ can be considered to be a minor-scope open problem although, as said above, Theorem \ref{othermethod} renders such task unnecessary when it comes to tackling particular dynamical systems, especially if computational ease is a factor and redundant terms arising from products and powers of formal first integrals are to be discarded.  

Ever since it first took form in \cite{ABSW}, the use of linearized higher variational equations to characterize, obstruct or redefine integrability is proving itself to be a welcome addition to the domain of dynamical systems. This contribution arrives at a time in which other attempts (e.g. \cite{MitsopoulosTsamparlis1,MitsopoulosTsamparlisUkpong1}) at finding methods to \emph{integrate} dynamical systems rather than proving them non-integrable are gradually gaining traction. 

A number of facts are germane to further studies:
\begin{ien}
\item the potential shown by higher variational equations in finding adequate changes of variables that will simplify the system into tractability and, potentially, integrability (this is most exemplified in the sample systems tackled in \S \ref{examples});
\item linked to the above, the possible application of a \emph{Baker-Campbell-Hausdorff} \cite{Dynkin} sorts of formula upon variable transformations, using the decomposition of a transformed fundamental matrix into a product of matrices to our advantage;
\item the possible amenability of results in \S \ref{filtersec} above to machine learning \cite{Montavon};
\item the applicability of this automatic quadrature-algebraic method to the application of generalized, multivariate \emph{Pad\'e approximants} \cite{Cuyt,GuillaumeHuard} to first integrals, which will be tackled in an immediate future;
\item \label{nham} the fact that the autonomous system \eqref{sys} need not be Hamiltonian, thus allowing any arbitrary system to be considered and losing the serious constraint posed by symplectic transformations, since variable transformations can now be chosen freely;
\item also related to \ref{nham}: the fact that this overrides the obstacles posed by the direct study of variational systems, which only has an implementation to Hamiltonian systems and requires first integrals to be meromorphic;
\item the potential to generalize \emph{Darboux integrability} (see e.g. \cite{dumortier}) to a wider set of ansatze, as already hinted at by the form taken in \eqref{fidix} or \eqref{sfir1};
\item and finally, the potential (mentioned at the start of \S \ref{filtersection} and already opened in \ref{SimonLVE}) for categorical functoriality translating the study of general dynamical systems to that of linear operators.
\end{ien}

\end{document}